\documentclass[a4paper,draft,11pt]{article}

\usepackage{pb-diagram,amsmath,amssymb}

\usepackage{amsthm}
\usepackage[T2A]{fontenc}
\usepackage[cp1251]{inputenc}
\usepackage[active]{srcltx}
\usepackage[english]{babel}

\usepackage{amsthm}

\newtheorem{theorem}{Theorem}
\newtheorem{definition}[theorem]{Definition}
\newtheorem{lemma}[theorem]{Lemma}

\newtheorem{proposition}[theorem]{Proposition}
\newtheorem{corollary}[theorem]{Corollary}

\def\myarrow{\ \hbox to 2em{\leaders
\hbox to 0.5ex{\hss\raise 0.55ex\hbox to 0.3ex{\hrulefill}\hss}
\hfill\,\llap{$>$}}\ }

\title{  Co-abelian toroidal compactifications  of  torsion free ball quotients  }
\author{   Azniv Kirkor Kasparian
\footnote{
Research partially supported by    Contract 178 / 09. 05. 2012 with Kliment Ohridski University of Sofia.  } }
\date{      }

\begin{document}
\maketitle

\thispagestyle{empty}

\begin{abstract}
Let $X'= ( {\mathbb B} / \Gamma )'$ be the toroidal compactification of the quotient ${\mathbb B} / \Gamma$ of the complex $2$-ball ${\mathbb B}$ by a torsion free lattice $\Gamma < SU_{2,1}$.
We say that ${\mathbb B} / \Gamma$ and $X'$ are co-abelian if there is an abelian surface, birational to $X'$.
The present work can be viewed as an illustration for the presence of a plenty of non-compact  co-abelian torsion free ball quotients ${\mathbb B} / \Gamma$.
More precisely, it shows that all the admissible  values $\frac{8 \pi ^2}{3} n \in  \frac{8 \pi ^2}{3} {\mathbb N}$ for  the volume of a quotient ${\mathbb B} / \Gamma$ by a torsion free lattice $\Gamma < SU_{2,1}$ are attained by co-abelian Picard modular  ${\mathbb B} / \Gamma_n$ over ${\mathbb Q}( \sqrt{-3})$,
 ${\rm vol} ( {\mathbb B} / \Gamma _n) = \frac{8 \pi ^2}{3} n$.
The article provides three types of infinite series $\mu '_n : ( {\mathbb B}  / \Gamma _n)' \rightarrow ({\mathbb B} / \Gamma _{n-1})'$ of finite unramified coverings $\mu'_n$ with abelian Galois groups, relating co-abelian torsion free Picard modular  toroidal compactifications $({\mathbb B} / \Gamma _n)'$ over ${\mathbb Q}( \sqrt{-3})$, with infinitely increasing volumes.
The first type is supported by mutually birational members $({\mathbb B} / \Gamma _n)'$ with fixed number of cusps.
The second kind is with mutually birational $({\mathbb B} / \Gamma _n)'$ and infinitely increasing number of cusps.
The third kind of series $\{ \mu '_n \} _{n=1} ^{\infty}$ relates mutually non-birational $({\mathbb B} / \Gamma _n)'$ with infinitely increasing number of cusps.
\end{abstract}

\section{Introduction}

Let
 $
{\mathbb B} = \{ ( z_1, z_2) \in {\mathbb C}^2 \ \ \vert \ \  | z_1| ^2 + | z_2|^2 < 1 \}
$
be the complex $2$-ball and $\Gamma$ be a discrete subgroup of $SU_{2,1}$.
The quotient $SU_{2,1} / \Gamma$ has finite invariant volume exactly when ${\mathbb B} / \Gamma$ has finite invariant volume.
If so, we say that $\Gamma$ is a lattice of $SU_{2,1}$.
The torsion elements of a lattice $\Gamma < SU_{2,1}$ are the ones of finite order, which are different from the identity.
The lattices $\Gamma$  without  torsion elements are said to be torsion free.

In \cite{AMRT}, Ash-Mumford-Rapoport-Tsai construct smooth toroidal compactifications of torsion free arithmetic quotients of bounded symmetric domains.
 Mok's \cite{Mok}, Hummel's \cite{Hummel}   and  Hummel-Schroeder's \cite{HummelSchroeder}  extend the construction to arbitrary  (not necessarily arithmetic) torsion free  lattices. .

\begin{definition}    \label{CoAbelianTorsionFreeTaroidalCompactification}
The torsion free ball quotient ${\mathbb B} / \Gamma$ and its  toroidal compactification $X'= \left( {\mathbb B} / \Gamma \right)'$ are co-abelian,
 if the minimal model    $X$  of $X'$ is an abelian surface.
\end{definition}

In \cite{WeakHolzConj} was shown that any ball quotient ${\mathbb B} / \Gamma$, birational to an abelian surface is smooth and non-compact.

The study of the co-abelian torsion free toroidal compactifications is initiated by Holzapfel in \cite{Ho04}.
Let $\xi : X'= \left( {\mathbb B} / \Gamma \right)'  \rightarrow  X$ be the blow-down of the smooth rational $(-1)$-curves on a torsion free toroidal compactification $X'= \left( {\mathbb B} / \Gamma \right)'$ to its   minimal model $X$ and $D = \xi (T) \subset X$ be the image of the toroidal compactifying divisor $T = \left( {\mathbb B} / \Gamma \right)'  \setminus  \left( {\mathbb B} /   \Gamma \right)$.
By \cite{AMRT}, $T = \sum\limits _{i=1} ^h T_i$ is a union of disjoint, contractible, smooth irreducible elliptic curves $T_i$.

\begin{definition}    \label{EllipticConfiguration}
A divisor $D = \sum\limits _{i=1} ^h D_i$ on a surface $X$ is called an elliptic configuration, if all the irreducible components $D_i$ of $D$ are smooth elliptic curves.
\end{definition}

The toroidal compactifying divisor $T = \sum\limits _{i=1} ^h T_i$ and its image $D = \xi (T) =  \sum\limits _{i=1} ^h \xi (T_i)$ are elliptic configurations on the smooth surfaces $X'$, respectively, $X$.
 The singular locus $D^{\rm sing} = \sum\limits _{1 \leq i < j \leq h} D_i \cap D_j$  of an elliptic configuration $D = \sum\limits _{i=1} ^h D_i$ consists of the intersection points of the different components.

\begin{definition}     \label{IntersectingEllipticConfiguration}
An elliptic configuration $D = \sum\limits _{i=1} ^n D_i$  is intersecting if it has at least one singular point.
\end{definition}

The toroidal compactifying divisors $T$ consist of disjoint smooth elliptic curves and are always non-intersecting.

An arbitrary torsion free toroidal compactification $X'= \left( {\mathbb B} / \Gamma \right)'$ with minimal model $X$ can be obtained from $X$ by blowing up   the singular points $D^{\rm sing}$ of $D$.
Thus, the pairs $ \left( X'= \left( {\mathbb B} / \Gamma \right)', T \right)$  and $(X,D)$  are in a bijective correspondence.

In \cite{BSA} Holzapfel has established that any torsion free toroidal compactification $X'= \left( {\mathbb B} / \Gamma \right)'$ with minimal model $X$ is subject to the proportionality condition
\begin{equation}    \label{ProportionalityCondition}
3 c_2 (X) - c_1 ^2 (X) = \sum\limits _{i=1} ^h K_X.D_i + \sum\limits _{i=1} ^h | D_i \cap D^{\rm sing} | - 4 | D^{\rm sing}|
\end{equation}
on the Chern numbers $c_2 (X)$, $c_1 ^2(X)$ of $X$ and the elliptic configuration $D = \xi (T) \subset X$.
An example of Momot from \cite{Momot2} illustrates that for $\kappa (X') = \kappa (X) =1$ the elliptic configuration $D$ is not supposed to be intersecting.
In other words, a torsion free toroidal compactification $X'= \left( {\mathbb B} / \Gamma \right)'$  of Kodaira dimension $\kappa (X') =1$ can be a minimal surface.

In the case of an abelian surface $X$, the canonical bundle $\mathcal{K} _X = \mathcal{O}_X$ is trivial, so that $K_X. D_i=0$ for any curve $D_i$ on $X$.
By the adjunction formula, $D_i^2=0$ and $D_i$ are not contractible.
Therefore any co-abelian torsion free toroidal compactification $X'= ( {\mathbb B} / \Gamma )'$ contain a smooth rational $(-1)$-curve and the elliptic configuration $D$ is always intersecting.
 Moreover, $c_2 (X) = c_1 (X) =0$ vanish and
 the proportionality of the co-abelian $X'= \left( {\mathbb B} /  \Gamma \right)'$ reads as
\begin{equation}   \label{AbelianProportionality}
\sum\limits _{i=1} ^h | D_i \cap D^{\rm sing}| = 4 |D^{\rm sing}|.
\end{equation}

 Recall that an isogeny of abelian varieties is an epimorphism with finite kernel.

\begin{theorem}     \label{HolzapfelsResults_On_CoAbelianTFTC}
{\rm (Holzapfel)}
(i) {\rm (Theorem 2.5 \cite {Ho04})} The blow-up $A'$ of an abelian surface $A$ at the singular points $D^{\rm sing}$ of an intersecting elliptic configuration $D = \sum\limits _{i=1} ^h D_i \subset A$ is a torsion free toroidal compactification $A'= \left( {\mathbb B} / \Gamma \right)'$ of a ball quotient ${\mathbb B} / \Gamma$ if and only if $D$ satisfies the proportionality condition (\ref{AbelianProportionality}).

(ii) {\rm (Corollary 2.8   \cite{Ho04}) }
The abelian minimal model $A$ of a torsion free toroidal compactification $\left( {\mathbb B} / \Gamma \right)'$ is isogeneous to the Cartesian square
$E \times E$ of an elliptic curve $E$.

 (iii) {\rm (Proposition 2.6   \cite{Ho04})} The isogenies of abelian surfaces pull back proportional elliptic configurations to proportional elliptic configurations.
Equivalently, if $\xi : X'= \left( {\mathbb B} / \Gamma \right)'\rightarrow X$ is the blow-down of the smooth rational $(-1)$-curves to the abelian minimal model $X$ of $X'$ and $\mu : Y \rightarrow X$ is an isogeny of abelian surfaces, then   the fibered product $Y' := Y \times _X X'$, defined by  the  commutative diagram
\[
\begin{diagram}
\node{Y}  \arrow{s,r}{\mu}   \node{Y' := Y \times _X X'} \arrow{w,t}{{\rm pr}_1}   \arrow{s,r}{{\rm pr}_2}  \\
\node{X}   \node{X'}   \arrow{w,t}{\xi}
\end{diagram}
\]
 is a torsion free toroidal compactification  $Y'= \left( {\mathbb B} / \Gamma _o \right)'$.
  The first canonical projection ${\rm pr} _1 : Y'\rightarrow Y$ is the blow-down of the smooth rational $(-1)$-curves on $Y'$ and the second canonical projection ${\rm pr} _2 : Y'\rightarrow X'$ is unramified $\ker ( \mu )$-Galois covering.
In particular, ${\rm pr}_2$ is a finite unramified abelian Galois covering, i.e., a finite unramified Galois covering
 with abelian Galois group $( \ker ( \mu), +)$.
 The lattice $\Gamma _o$ is a subgroup of $\Gamma$ of finite index.

From now on, we refer briefly to $Y':= Y \times _X X'$ as to an isogeny pull back of $X'= \left( {\mathbb B} / \Gamma \right)'$.
\end{theorem}

In order to recall the notion of a Picard modular torsion free toroidal compactification $X'= ( {\mathbb B} / \Gamma)'$,
let us  consider an imaginary quadratic number field ${\mathbb Q} ( \sqrt{-d})$ with integers ring $\mathcal{O}_{-d}$
and denote by $M_{3 \times 3} ( \mathcal{O}_{-d})$ the set of  $3 \times 3$-matrices with entries from $\mathcal{O}_{-d}$.
If $\mathcal{O}_{-d}^*$ is the units group of $\mathcal{O}_{-d}$, then the group
\[
SU_{2,1} ( \mathcal{O}_{-d}) := \{ g \in SU_{2,1} \cap M_{3 \times 3} ( \mathcal{O}_{-d}) \ \ \vert \ \  g^{-1} \in M_{3 \times 3} ( \mathcal{O}_{-d}) \} =
\]
\[
= \{ g \in SU_{2,1} \cap   M_{3 \times 3} ( \mathcal{O}_{-d}) \ \ \vert \ \  \det (g) \in \mathcal{O}_{-d}^* \}.
\]
The lattices $\Gamma _1$ and $\Gamma _2$ of $SU_{2,1}$ are commensurable, if their intersection $\Gamma _1 \cap \Gamma _2$ is of finite index in $\Gamma _1$ and $\Gamma _2$.

\begin{definition}   \label{PicardModularTorroidalCompactification}
The lattice $\Gamma$ toroidal compactification $X'= ( {\mathbb B} / \Gamma)'$ are called Picard modular over the imaginary quadratic number field
 ${\mathbb Q}( \sqrt{-d})$, if the lattice $\Gamma$ is commensurable with $SU_{2,1} ( \mathcal{O}_{-d})$ for the integers ring $\mathcal{O}_{-d}$ of
  ${\mathbb Q}( \sqrt{-d})$..
\end{definition}

Any non-uniform arithmetic lattice $\Gamma < SU_{2,1}$ is Picard modular over some  imaginary quadratic number field ${\mathbb Q}( \sqrt{-d})$ (cf. \cite{EmeryStover}).

Here is a brief synopsis of the paper.
The next section establishes that for any torsion free lattice $\Gamma < SU_{2,1}$ there is a torsion free lattice $\Gamma _o$ with co-abelian quotient ${\mathbb B} / \Gamma _o$, such that ${\rm vol} ( {\mathbb B} / \Gamma ) = {\rm vol} ( {\mathbb B} / \Gamma _o)$ with respect to the Haar measure of $SU_{2,1}$.
There follow two technical sections. More precisely, the third one studies the elliptic curves on a product $E_1 \times E_2$ of elliptic curves $E_1$, $E_2$.
The fourth section describes the pull-backs of elliptic curves under diagonal isogenies $\mu = \left(  \begin{array}{cc}
\alpha  &  0  \\
0  &  \beta
\end{array}  \right) : E_2 \times E_2  \rightarrow E_1 \times E_1$ of Cartesian squares $E_i \times E_i$ of elliptic curves $E_i$.
The final, fifth section  provides explicit isogeny series  of co-abelian torsion free Picard modular toroidal compactifications over ${\mathbb Q}( \sqrt{-3})$, with infinitely increasing volumes.
There are ones with mutually birational terms and fixed number of cusps, as well as ones with mutually birational terms and infinitely increasing number of cusps. After showing that $E_2 \times E_2$ and $E_1 \times E_1$ with $End(E_1) \neq End(E_2)$  are not birational,  the article constructs an infinite isogeny series of mutually non-birational co-abelian torsion free toroidal compactifications with infinitely increasing volumes and infinitely increasing number of cusps.

\section{The  volumes of torsion free ball quotients are attained by  co-abelian  torsion free Picard modular  ones over ${\mathbb Q}(\sqrt{-3})$}

Let us recall some results on the volumes of the local complex hyperbolic surfaces ${\mathbb B} / \Gamma$.

\begin{theorem}    \label{HirzebruchVolume}
{\rm (Hirzebruch \cite{Hir1},  \cite{Hir3})} Let
\[
{\mathbb B}^n = \{ (z_1, \ldots , z_n) \in {\mathbb C}^n \ \ \vert \ \  |z_1|^2 + \ldots + |z_n|^2 = 1 \}
\]
be the complex $n$-dimensional ball, $\Gamma$ be a lattice of $SU_{n,1}$ and $e( {\mathbb B}^n / \Gamma)$ be the Euler number of ${\mathbb B}^n / \Gamma$.
Then the volume of ${\mathbb B}^n / \Gamma$ with respect to the Haar measure is
\[
{\rm vol} ( {\mathbb B} ^n / \Gamma  ) = \frac{(- \pi) ^n 2^{2n}}{(n+1)!} e \left( {\mathbb B}^n / \Gamma \right).
\]
\end{theorem}

From now on, we denote by ${\mathbb B} := {\mathbb B}^2$ the complex $2$-dimensional ball and note that for a torsion free lattice $\Gamma$ of $SU_{2,1}$, the Euler number $e( {\mathbb B} / \Gamma) \in {\mathbb Z}$ is integral. Combining with ${\rm vol} ( {\mathbb B} / \Gamma ) = \frac{8 \pi ^2}{3} e ( {\mathbb B} / \Gamma ) \in {\mathbb R}^{>0}$, one concludes that $e( {\mathbb B} / \Gamma ) \in {\mathbb N}$ and
\begin{equation}    \label{Dim2HirVolFormula}
{\rm vol} ( {\mathbb B} / \Gamma ) = \frac{8 \pi ^2}{3} e ( {\mathbb B} / \Gamma ) \in  \frac{8 \pi^2}{3} {\mathbb N}.
\end{equation}

\begin{theorem}   \label{HersonskyPaulin}
{\rm (Hersonsky and Paulin \cite{HP})}
The smallest volume of a compact quotient ${\mathbb B} / \Gamma$ by a torsion free lattice $\Gamma < SU_{2,1}$ is $8 \pi ^2$.
\end{theorem}

\begin{theorem}   \label{Parker}
{\rm (Parker \cite{P})}
The smallest volume of a non-compact   quotient ${\mathbb B} / \Gamma$ by a  torsion free  lattice $\Gamma < SU_{2,1}$ is $\frac{8 \pi ^2}{3}$.
\end{theorem}

\begin{lemma}   \label{ECATFTCIsNum-1Curves}
Suppose that the blow-down
\[
\xi : (X'= ( {\mathbb B} / \Gamma , T = X'  \setminus  ( {\mathbb B} / \Gamma))  \longrightarrow (X, D = \xi (T))
\]
  of the smooth rational $(-1)$-curves on a  torsion free toroidal compactification  produces the abelian surface $X$ with the proportional elliptic configuration $D$.
  Then the Euler number $e( {\mathbb B} / \Gamma)$ of ${\mathbb B} / \Gamma$ equals the number $s$ of the smooth rational $(-1)$-curves on $X'$
  and the number of the singular points of  $D$.
\end{lemma}

\begin{proof}

 The Euler number $e( {\mathbb B} / \Gamma ) = e( {\mathbb B} / \Gamma )'$ since the toroidal compactifying divisor $T = ( {\mathbb B} / \Gamma )'\setminus ({\mathbb B} / \Gamma )$ of ${\mathbb B} / \Gamma$ consists of smooth irreducible elliptic curves $T_i$ and $e(T_i)=0$.
 Recall that the abelian minimal model $X$ of $X'= ( {\mathbb B} / \Gamma )'$ has vanishing Euler number $e(X)=0$ and $X'$ is obtained from $X$ by blowing up $s$ points.
 Therefore $e(X') = e(X) + s = s$ and $e( {\mathbb B} / \Gamma ) = s$.

\end{proof}

Let $\mathcal{O}_{-3} = {\mathbb Z} + \omega _{-3} {\mathbb Z}$ with $\omega _{-3} = e^{\frac{ \pi i}{3}}$ be the ring of Eisenstein integers, i.e., the integers ring of the imaginary quadratic number field ${\mathbb Q}( \sqrt{-3})$.
Consider the elliptic curve $E_{-3} = {\mathbb C} / \mathcal{O}_{-3}$ and the abelian surface $A_{-3} = E_{-3} \times E_{-3}$.
In \cite{Hir84}  Hirzebruch constructs a series $\{ X_n \} _{n=1} ^{\infty}$ of minimal surfaces of general type with
 $\lim\limits _{n \to \infty} \frac{c_1^2 (X_n)}{c_2(X_n)} = 3$.
 The surfaces $X_n$ are birational to branched covers of $A_{-3}$, ramified over an elliptic configuration $D_{-3} ^{\rm Hir} = \sum\limits _{i=1} ^4 D_i ^{\rm Hir} \subset A_{-3}$ with a single singular point $\check{o}_{A_{-3}}$.
 In \cite{Ho04} Holzapfel shows that $D_{-3} ^{\rm Hir}$ is a proportional elliptic configuration on $A_{-3}$, so that the blow-up $A_{-3}'$ of $A_{-3}$ at the origin $\check{o}_{A_{-3}}$ is a torsion free toroidal compactification $A'_{-3} = ( {\mathbb B} / \Gamma _{-3} ^{\rm Hir} )'$ with a single smooth rational $(-1)$-curve.
 Moreover, in \cite{Ho86} Holzapfel proves that $A'_{-3} = ( {\mathbb B} / \Gamma _{-3} ^{\rm Hir} )'$ is Picard modular over ${\mathbb Q}( \sqrt{-3})$.
 By Lemma \ref{ECATFTCIsNum-1Curves} the non-compact co-abelian torsion free Hirzebruch's example ${\mathbb B} / \Gamma _{-3} ^{\rm Hir}$ has
  Euler number $e( {\mathbb B} / \Gamma _{-3} ^{\rm Hir}) =1$ and minimal volume ${\rm vol} ( {\mathbb B} / \Gamma _{-3} ^{\rm Hir}) = \frac{8 \pi ^2}{3}$.

In \cite{EmeryStover}, Emery and Stover express the minimal volume $v_{-d,n}$ of a quotient $PU_{n,1} / \Gamma  $ by a non-uniform Picard modular lattice
 $\Gamma < PU_{n,1}$ over ${\mathbb Q}( \sqrt{-d})$ by the  $L$-function $L_{-d} = \frac{\zeta _{-d}}{\zeta}$, associated with Dedekind zeta function $\zeta _{-d}$ of ${\mathbb Q}( \sqrt{-d})$ and the Riemann zeta function $\zeta$.
 They estimate the number of the isomorphism classes of the Picard modular lattices $\Gamma < PU_{n,1}$ over ${\mathbb Q}( \sqrt{-d})$  with minimal volume
  ${\rm vol} ( PU_{n,1} / \Gamma ) = v_{-d,n}$,
  For all $n \geq 2$, the non-uniform arithmetic lattices $\Gamma < PU_{n,1}$ of smallest ${\rm vol} ( PU_{n,1} / \Gamma) = v_{-d,n}$ are shown to be Picard modular over ${\mathbb Q}( \sqrt{-3})$.
  For an even $n = 2k$, there are exactly two isomorphism classes of non-uniform lattices of  $ PU_{2k,1}$ with minimal  co-volume $v_{-3,2k}$.
  For an odd $n = 2k+1 \not \equiv 7 ( {\rm mod} \ \ 8)$ there is a unique  isomorphism class of non-uniform lattices of $PU_{2k+1,1}$ with  minimal co-volume $v_{-3,2k+1}$.

  The previous work \cite{Stover} of Stover establishes that the minimal volume of a torsion  non-compact arithmetic quotient ${\mathbb B}^2 / \Gamma$ is
   $\frac{ \pi ^2}{27}$.
   Moreover, he proves that any torsion free arithmetic  ${\mathbb B}  / \Gamma$ of ${\rm vol} ( {\mathbb B} / \Gamma) = \frac{8 \pi^2}{3}$ covers (at least) one of the two non-isomorphic torsion Picard modular surfaces ${\mathbb B} ^2 / \Gamma _{-3}$ or ${\mathbb B}^2 / \Gamma _{-3}'$ over ${\mathbb Q}( \sqrt{-3})$  with   ${\rm vol} ( {\mathbb B}^2 / \Gamma _{-3}) = {\rm vol} ( {\mathbb B}^2 / \Gamma '_{-3}) = \frac{\pi ^2}{27}$.
   Stover's \cite{Stover}  provides a complete list of representatives of the isomorphism classes of torsion free $\Gamma < \Gamma _{-3} \cap \Gamma '_{-3}$ with ${\rm vol} ( {\mathbb B} / \Gamma ) = \frac{8 \pi ^2}{3}$.

   The   non-compact   quotients ${\mathbb B}^2 / \Gamma$  by torsion  non-arithmetic lattices $\Gamma < SU_{2,1}$ are not expected to be of minimal volume $\frac{\pi ^2}{27}$.

\begin{theorem}    \label{AllVolumes}
For any admissible  value $\frac{8 \pi ^2}{3}m \in \frac{8 \pi ^2}{3} {\mathbb N}$ of the volume of a quotient ${\mathbb B} / \Gamma$ of ${\mathbb B}$ by a torsion free lattice $\Gamma < SU_{2,1}$,  there is a non-compact, co-abelian,  torsion free, Picard modular  ${\mathbb B} / \Gamma _{-3,m}$ over ${\mathbb Q}( \sqrt{-3})$, with volume
 \[
 {\rm vol} ( {\mathbb B} / \Gamma _{-3,m} ) = \frac{8 \pi ^2}{3}m.
 \]
\end{theorem}

\begin{proof}

The co-abelian torsion free toroidal compactification $X' _{-3}=  ( {\mathbb B} / \Gamma _{-3} ^{\rm Hir})'$ has fundamental group
$\pi _1 (X'_{-3}) \simeq ( \pi _1 (A_{-3}, +) = ( \mathcal{O}_{-3} \times \mathcal{O}_{-3}, +)$.
For an arbitrary $m \in {\mathbb N}$ consider the order $\mathcal{O}_{-3,m} = {\mathbb Z} + m \mathcal{O}_{-3}$ of ${\mathbb Q}( \sqrt{-3})$ with conductor
$m=[ \mathcal{O}_{-3} : \mathcal{O}_{-3,m} ]$ and the elliptic curve $E_{-3,m} = {\mathbb C} / \mathcal{O}_{-3,m}$.
Then
\[
\mu _m  : A_{-3,m} := E_{-3} \times E_{-3,m}  \longrightarrow A_{-3} = E_{-3} \times E_{-3},
\]
\[
\mu _m ( u + \mathcal{O}_{-3}, v + \mathcal{O}_{-3,m} )  = (u + \mathcal{O}_{-3}, v + \mathcal{O}_{-3})
\]
is an isogeny of degree $m$.
If $\xi _{-3} : X'_{-3}  \rightarrow A_{-3}$ is the blow-down of the smooth rational $(-1)$-curve on $X'_{-3}$ then the fibered product
\[
X'_{-3,m} := A_{-3,m}  \times _{A_{-3}}  X'_{-3} = ( {\mathbb B} / \Gamma _{-3,m} )'
 \]
 is a torsion free  Picard modular toroidal compactification over ${\mathbb Q}( \sqrt{-3})$, with abelian minimal model $A_{-3,m}$ and the canonical projection ${\rm pr} _2 : X'_{-3,m}    \rightarrow X'_{-3}$ is an unramified covering of degree $m$.
Therefore
\[
{\rm vol} ( {\mathbb B} / \Gamma _{-3,m}) = \deg ({\rm pr_2}) {\rm vol} ( {\mathbb B} / \Gamma _{-3} ^{\rm Hir})  = \frac{8 \pi^2}{3} m.
\]

\end{proof}



\section{Elliptic curves on abelian surfaces}

In order to study the co-abelian torsion free toroidal compactifications $X'=  ( {\mathbb B} / \Gamma )'$ through their associated proportional elliptic configurations, let us make some considerations, concerning the elliptic curves on abelian surfaces.

Let $F$ be an elliptic curve on an abelian surface $A$.
If $\check{o}_F$ is the origin of $F$, then the  translation $\tau ( - \check{o}_F): A \rightarrow A$ by $ - \check{o}_F \in A$ transforms $F$
 into the elliptic curve $F^{o}$ with origin $\check{o}_{F^{o}} = \check{o}_A$.
The identical inclusion ${\rm Id} : (F^{o}, +) \hookrightarrow (A, +)$ is a group homomorphism.
Consider the universal covering $U_{F^{o}} : \widetilde{F^{o}} = ({\mathbb C},+)   \rightarrow (F^{o}, +)$ of $F^{o}$ and fix the point
 $0  \in U^{-1} _{F^{o}} ( \check{o}_{F^{o}} ) = U^{-1} _{F^{o}} ( \check{o_{A}})$.
 Similarly, let $U_A : \widetilde{A} = ({\mathbb C}^2,+) \rightarrow (A,+)$ be the universal covering of $A$ and fix the point $(0,0) \in U_A^{-1} ( \check{o}_A)$.
 Since $U_A$ is unramified and
 $({\rm Id} \circ U_{F^{o}} ) _* \pi _1 ( \widetilde{F^{o}} ) = \{ 1_{\pi _1 (A)} \} = (U_A) _* \{ 1 \} = (U_A)_* \pi _1 ( \widetilde{A}) $,
  the   holomorphic map  ${\rm Id} \circ  U_{F^{o}} : \widetilde{F^{o}} = {\mathbb C} \rightarrow A$ admits an  affine linear lifting
  $\lambda : \widetilde{F{o}} = ( {\mathbb C},+)  \rightarrow \widetilde{A} = ( {\mathbb C}^2, +)$ with $\lambda (0) = (0,0)$.
In other words, there is a ${\mathbb C}$-linear map $\lambda : \widetilde{F^{o}}   \rightarrow \widetilde{A}$, closing the  commutative diagram
 \begin{equation}   \label{LinearLifting}
 \begin{diagram}
 \node{{\mathbb C}}  \arrow{s,r}{U_{F^{o}}}  \arrow{e,t}{\lambda}   \node{{\mathbb C}^2}  \arrow{s,r}{U_A}   \\
 \node{F^{o}}  \arrow{e,t}{{\rm Id}}   \node{A}
 \end{diagram}.
 \end{equation}
If $\lambda (1) = (a,b) \in {\mathbb C}^2$ then $(a,b) \neq (0,0)$ and the line
\[
L(a,b) := \lambda ( {\mathbb C}) =  \{ ( at, bt) \ \ \vert \ \ t \in {\mathbb C} \} \subset {\mathbb C}^2  = \widetilde{A}
\]
through the origin $(0,0) \in \widetilde{A}$ covers the elliptic curve
\[
(F^{o}, +) = \left( L(a,b) / L(a,b) \cap \pi _1 (A), + \right) \simeq ( L(a,b) + \pi _1 (A) / \pi _1 (A), + ) \leq (A,+)
\]
through the origin $\check{o}_A$.
 From now on, we refer to   $(a,b) \in {\mathbb C}^2 \setminus \{ (0,0) \}$ as to  the slope vector of $F^{o} = \{ (at, bt) + \pi _1 (A) \ \ \vert \ \  t \in {\mathbb C} \} / \pi _1(A)$ and $F  = F^{o} + \check{o}_F$.

The non-zero ${\mathbb C}$-linear map $\lambda : {\mathbb C} \rightarrow {\mathbb C}^2$ is a ${\mathbb C}$-linear embedding and restricts to an embedding
\[
\lambda : (\pi _1 (F) = \pi _1 (F^{o}), +)  \longrightarrow (L(a,b), +)
\]
of ${\mathbb Z}$-modules.
The induced homomorphism ${\rm Id} _* : \pi _1 (F^{o}) \rightarrow \pi _1 (A)$ of the fundamental groups coincides with $\lambda$ and
$\lambda ( \pi _1 (F^{o})) = L(a,b) \cap \pi _1 (A)$.
The commutative diagram (\ref{LinearLifting}) reduces to
\[
\begin{diagram}
\node{{\mathbb C} }  \arrow{s,l}{U_{F^{o}}}   \arrow{e,t}{\lambda}   \node{L(a,b)}   \arrow{sw,r}{\zeta _{a,b}}   \\
\node{F^{o}}  \node{\mbox{  }}
\end{diagram}
\]
for the $L(a,b) \cap \pi _1 (A)$-Galois covering $\zeta _{a,b} : (L(a,b), +) \rightarrow (F^{o},+)$.

If $A = E_1 \times E_2$ is a product of elliptic curves and $ab \neq 0$ then the isomorphic image $\lambda ( \pi _1 (F^{o}))$ of $\pi _1 (F) = \pi _1( F^{o})$ is
\[
\lambda ( \pi _1 ( F^{o})) = L(a,b) \cap ( \pi _1 (E_1) \times \pi _1(E_2)) =
\]
\[
= \{ (at, bt) \in {\mathbb C}^2  \ \  \vert \ \  t \in a^{-1} \pi _1 (E_1) \cap b^{-1} \pi _1 (E_2) \},
\]
so that
\[
\pi _1 (F) = \pi _1 (F^{o}) = a^{-1} \pi _1 (E_1) \cap b^{-1} \pi _1 (E_2).
\]
In the case of $b=0$ one has
\[
\lambda ( \pi _1 (F^{o})) = \{ (at, 0) \in {\mathbb C}^2 \ \ \vert \ \  t \in a^{-1} \pi _1 (E_1) \},
\]
whereas
\[
\pi _1 (F) = \pi _1 (F^{o}) = a^{-1} \pi _1 (E_1).
\]

Note that
\[
L(a,b) = \{ (u,v) \in {\mathbb C}^2 \ \ \vert \ \  bu - av = 0 \}
\]
and the complete preimage
\[
U_A^{-1} ( F^{o})  = L(a,b) + ( \pi _1 (E_1) \times \pi _1 (E_2)) =
\]
\[
= \{ (u,v) \in {\mathbb C}^2 \ \ \vert \ \  bu - av \in a \pi _1 (E_2) + b \pi _1 (E_1) \}.
\]
As far as $(a \pi _1 (E_2) + b \pi _1 (E_1), +) \geq (a \pi _1 (E_2),+) \simeq ({\mathbb Z}^, +)$ is a ${\mathbb Z}$-submodule of $({\mathbb C}, +)$ of rank $\geq 2$, the quotient $F':= ( {\mathbb C}, +) / ( a \pi _1 (E_2) + b \pi _1 (E_1), +)$ is a smooth elliptic curve.
The map
\[
\psi _{(a,b)} : E_1 \times E_2 \longrightarrow F'
\]
\[
\psi _{(a,b)} ( u + \pi _1 (E_1), v + \pi _1 (E_2)) = bu - av + a \pi _1 (E_2) + b \pi _1 (E_1)
\]
 is a homomorphism of abelian varieties with kernel
 $\ker ( \psi _{(a,b)} ) = F^{o}$.
 The fibres of $\psi _{(a,b)}$ are elliptic curves on $E_1 \times E_2$, parallel to $F$.
 More precisely, for any point $(P,Q) = ( p + \pi _1 (E_1), q + \pi _1 (E_2)) \in E_1 \times E_2$, the fibre of $\psi _{(a,b)}$ through $(P,Q)$ is
 \[
 \psi ^{-1}_{(a,b)} \psi _{(a,b)} (P,Q) = \{ ( u + \pi _1 (E_1), v + \pi _1 (E_2)) \ \ \vert \ \  b (u-p) -a (v-q) + a \pi _1 (E_2) + b \pi _1 (E_1) \}.
 \]
 Note that the elliptic curve $F^{o}$ does not depend on the slope vector $(a,b)$ but on its class  of proportionality    $\{ (ac, bc) \in {\mathbb C}^2 \ \ \vert \ \  c \in {\mathbb C}^* \}$.
 Altogether, we have proved the following

 \begin{proposition}  \label{EllipticFibration}
 An arbitrary elliptic curve $F$ through the origin on the Cartesian product $E_1 \times E_2$ of elliptic curves $E_1$, $E_2$ is of the form
 \[
 F = \{ (at + \pi _1 (E_1), bt + \pi _1 (E_2) ) \in E_1 \times E_2 \ \ \vert \ \  t \in {\mathbb C} \}
 \]
 for some slope vector $(a,b) \in {\mathbb C}^2 \setminus \{ (0,0) \}$.
 The fundamental group
 \[
 \pi _1 (F) =
 \begin{cases}
a^{-1} \pi _1 (E_1) \cap b^{-1} \pi _1 (E_2)  &  \text{ for $a \neq 0$, $b \neq 0$,   }   \\
 a^{-1}  \pi _1 (E_1)  &  \text{ for $a \neq 0$, $b=0$,}    \\
 b^{-1} \pi _1 (E_2)   &  \text{  for $a=0$, $b \neq 0$.}
 \end{cases}
 \]
 There is a natural fibration
 \[
 \psi _{(a,b)} : (E_1 \times E_2, +)   \longrightarrow     F'= ( {\mathbb C}, +) / ( a \pi _1 (E_2) + b \pi _1 (E_1), +),
 \]
 \[
 \psi _{(a,b)} (u + \pi _1 (E_1), v + \pi _1 (E_2) ) = bu - av + a \pi _1 (E_2) + b \pi _1 (E_1)
 \]
 of $E_1 \times E_2$ by elliptic curves, parallel to
 \[
 F = \ker ( \psi _{(a,b)}  ) = \{ ( u + \pi _1 (E_1), v + \pi _1 (E_2)) \in E_1 \times E_2  \ \  \vert \ \  bu - av \in a \pi _1 (E_2) + b \pi _1 (E_1) \}.
 \]
 \end{proposition}

We are going to consider Cartesian squares $E \times E$ of elliptic curves $E$.
From now on, let us denote by
\[
E(a,b) := \{ (at + \pi _1 (E), bt + \pi _1 (E) ) \ \ \vert \ \  t \in {\mathbb C} \}=
\]
\[
= \{ (u + \pi _1 (E), v + \pi _1 (E)) \ \ \vert \ \  bu -av \in a \pi _1 (E) + b \pi _1 (E) \}
\]
the elliptic curve on $E \times E$ through the origin $\check{o}_{E \times E}$, with slope vector $(a,b) \in {\mathbb C}^2 \setminus \{ (0,0) \}$.

\section{Isogeny pull back of an elliptic curve}

 The isogenies of abelian surfaces  lift to ${\mathbb C}$-linear maps of the corresponding universal covers.
 In particular, one can identify the isogenies
 \[
 \mu : E_m \times E_m \longrightarrow E_1 \times E_1
 \]
 with their matrices $\mu \in GL(2, {\mathbb C})$.
 For simplicity, let us  restrict to diagonal
 \[
 \mu = \left( \begin{array}{cc}
 \alpha  &  0  \\
 0  &  \beta
 \end{array}   \right).
 \]

Our aim is to construct unramified      coverings $\mu ': Y'=  ( {\mathbb B} /  \Gamma _2 )'   \rightarrow X'= ( {\mathbb B} / \Gamma _1 )'$ of co-abelian torsion free toroidal compactifications, induced by isogenies $\mu : Y \rightarrow X$ of the corresponding minimal models.
If $D = \sum\limits _{i=1} ^h D_i \subset X$   is the image of the toroidal compactifying divisor
 $\sum\limits _{i=1} ^h T_i = T = ( {\mathbb B} / \Gamma _1 )'   \setminus ( {\mathbb B} / \Gamma _1)$ under the blow-down of the smooth rational  $(-1)$-curves on $X'$, then $\mu ^{-1} (D) = \sum\limits _{i=1} ^h \mu ^{-1} ( D_i) = \sum\limits _{i=1} ^h \sum\limits _{j=1} ^{l_i} D_{ij} \subset Y$ is the image of the toroidal compactifying divisor $( {\mathbb B} / \Gamma _2 )' \setminus ( {\mathbb B} / \Gamma _2)$ under the blow-down of the smooth rational $(-1)$-curves.
 The number of the irreducible components of $\mu ^{-1}(D)$  equals the number of the cusps of ${\mathbb B} / \Gamma _2$.

 The following lemma  provides a general expression for the number  $n( \mu ^{-1}(F))$ of the irreducible components of an isogeny pull back $\mu ^{-1} (F)$ of an elliptic curve $F$ on a Cartesian square of an elliptic curve.

 \begin{lemma}    \label{CoarseCuspLift}
 On the Cartesian square $E_1 \times E_1$ of an elliptic curve $E_1$, let us consider an elliptic curve $F = E_1 (a,b) + (P,Q) \subset E_1 \times E_1$ with slope vector $(a,b) \in {\mathbb C}^2 \setminus \{ (0,0) \}$ through $(P, Q)  = (p + \pi _1(E_1), q + \pi _1 (E_1)) \in E_1 \times E_1$
  for some $p,q \in {\mathbb C}$.
Assume that the elliptic curve $E_2$ and  the complex numbers $\alpha,  \beta \in {\mathbb C}^*$ are subject to
\[
\alpha \pi _1 (E_2) \subseteq \pi _1 (E_1), \ \ \beta \pi _1 (E_2) \subseteq \pi _1 (E_1).
\]
Then  the lattice $\Lambda (F, \mu ) := a \beta \pi _1 (E_2) + b \alpha \pi _1 (E_2)$ of $({\mathbb C}, +)$ is contained in the lattice
 $\Lambda (F) := a \pi _1 (E_1) + b \pi _1 (E_1)$ of $({\mathbb C}, +)$ and
\[
\mu = \left( \begin{array}{cc}
\alpha  &  0  \\
0  &  \beta
\end{array}   \right) : E_2 \times E_2 \longrightarrow E_1 \times E_1
\]
is an isogeny of degree
\[
\deg ( \mu ) = [ \pi _1 (E_1) : \alpha \pi _1 (E_2)] [ \pi _1 (E_1) : \beta \pi _1 (E_2)],
\]
which pulls back $F$ to a disjoint union
\begin{equation}   \label{ExplicitCoarsePreimage}
\mu ^{-1}(F) = \sum\limits _{a \lambda _1 + b \lambda _2 + \Lambda (F,\mu)} \left[ E_2 ( a \beta, b \alpha ) +
\left( \frac{p + \lambda _2}{\alpha} + \pi _1 (E_2), \frac{q - \lambda _1}{\beta} + \pi _1 (E_2) \right) \right]
\end{equation}
of $[ \Lambda (F) : \Lambda (F, \mu )]$ smooth irreducible elliptic curves, parallel to $E_2 ( a \beta , b \alpha )$.
 \end{lemma}

 \begin{proof}

The assumptions $\alpha \pi _1 (E_2) \subseteq \pi _1 (E_1)$ and $\beta \pi _1 (E_2) \subseteq \pi _1 (E_1)$ suffice for the diagonal
 $\mu : E_2 \times E_2 \rightarrow  E_1 \times E_1$ to be a homomorphism of abelian varieties.
Due to $\alpha \neq 0$, $\beta \neq 0$, the homomorphism $\mu$ is surjective and with a finite kernel.
In other words, $\mu$ is an isogeny.
According to
\[
\ker ( \mu ) = \left( \frac{1}{\alpha} \pi _1 (E_1) / \pi _1 (E_2) \right) \times \left( \frac{1}{\beta} \pi _1 (E_1) / \pi _1 (E_2) \right),
\]
 $\mu$  is of degree
\[
\deg ( \mu) = | \ker (\mu )| =  [ \pi _1 (E_1) : \alpha \pi _1 (E_2)][ \pi _1 (E_1) : \beta \pi _1 (E_2)].
\]
In order to describe the pull back $\mu ^{-1}(F)$ of $F = E_1 (a,b) + (P,Q)$ under $\mu$, let us note that $(P_o, Q_o) := \left( \frac{p}{\alpha} + \pi _1 (E_2), \frac{q}{\beta} + \pi _1 (E_2) \right) \in \mu ^{-1} (P,Q)$.
Straightforwardly,   $\mu ^{-1} (F) = \mu ^{-1} E_1 (a,b) + (P_o,Q_o)$ and   the description of $\mu ^{-1}(F)$ reduces to the description of the subgroup
 $( \mu ^{-1} E_1 (a,b), +) \leq ( E_2 \times E_2, +)$.

Consider the group homomorphism
\[
\psi _1 : (E_1 \times E_1, +) \longrightarrow ( E_1 (a,b)', +) = \left( {\mathbb C} / \Lambda (F), + \right),
\]
\[
\psi _1 ( u + \pi _1 (E_1), v + \pi _1 (E_1)) = bu - av + a \pi _1 (E_1) + b \pi _1 (E_1) =  bu - av + \Lambda (F)
\]
with kernel $\ker ( \psi _1) = (E_1(a,b), +)$.
The composition
\[
\psi _1 \circ \mu : (E_2 \times E_2, +) \longrightarrow (E_1(a,b)' = {\mathbb C} / \Lambda (F), +)
\]
\[
\psi _1 \circ \mu (u + \pi _1 (E_2), v + \pi _1 (E_2)) = b \alpha u - a \beta v + a \pi _1 (E_1) + b \pi _1 (E_1) = b \alpha u - a \beta v + \Lambda (F)
\]
is a group homomorphism with kernel
\[
\ker( \psi _1 \circ \mu) = ( \psi _1 \circ \mu ) ^{-1} ( \check{o}) = \mu ^{-1} \psi _1 ^{-1} ( \check{o}) =
\mu ^{-1} \ker( \psi _1) = \mu ^{-1} E_1 (a,b).
\]
Note that the group homomorphism
\[
\psi _2 : (E_2 \times E_2, +) \longrightarrow \left( E_2 (a \beta , b \alpha )'= {\mathbb C} /  \Lambda (F,\mu), + \right),
\]
\[
\psi _2 ( u + \pi _1 ( E_2), v + \pi _1 (E_2)) = b \alpha u - a \beta v + a \beta \pi _1 (E_2) + b \alpha \pi _1 (E_2) =
 b \alpha u - a \beta v + \Lambda (F,\mu)
\]
 has kernel $\ker (\psi _2) = E_2 (a \beta, b \alpha) \subseteq \ker( \psi _1 \circ \mu)$, due to  $\Lambda (F, \mu)  \subseteq  \Lambda (F)$.
 Therefore  $\psi _1 \circ \mu$ factors through $\psi _2$ and the natural epimorphism
\[
\nu : (E_2 (a \beta, b \alpha)', +) \longrightarrow (E_1 (a,b)', +),
\]
\[
\nu ( x + a \beta \pi _1 (E_2) + b \alpha \pi _1 (E_2)) = x + a \pi _1 (E_1) + b \pi _1 (E_1) \ \ \mbox{  for  } \ \ \forall x \in {\mathbb C}.
\]
In other words, there is a commutative diagram
\[
\begin{diagram}
\node{E_2 \times E_2} \arrow{e,t}{\mu}  \arrow{s,r}{\psi _2}  \node{E_1 \times E_1}  \arrow{s,r}{\psi _1}  \\
\node{E_2 ( a \beta , b \alpha)'}   \arrow{e,t}{\nu}  \node{E_1 (a,b)'}
\end{diagram}
\]
of homomorphisms of additive groups.
As a result,
\[
\mu ^{-1} E_1 (a,b) = ( \psi _1 \circ \mu) ^{-1} ( \check{o}) =
 ( \nu \circ \psi _2) ^{-1} ( \check{o})  =  \psi _2 ^{-1} \ker( \nu)
\]
consists of $| \ker(\nu)|$ smooth irreducible elliptic  components, parallel to $\ker ( \psi _2) = E_2( a \beta , b \alpha)$.
We claim that
\begin{equation}   \label{PulledBackKernel}
\psi _2 ^{-1} ( \ker( \nu)) = \sum\limits _{a \lambda _1 + b \lambda _2 + \Lambda (F, \mu)}  \left[ E_2 (a \beta, b \alpha) +
 \left( \frac{\lambda _2}{\alpha} + \pi _1 (E_2), - \frac{\lambda _1}{\beta} + \pi _1(E_2) \right) \right]
 \end{equation}
with a  summation  over all $a \lambda _1 + b \lambda _2 + \Lambda (F, \mu) \in \Lambda (F) / \Lambda (F, \mu) = \ker (\nu)$.
The inclusion  $E_2 (a \beta, b \alpha) + \left( \frac{\lambda _2}{\alpha} + \pi _1 (E_2), - \frac{\lambda _1}{\beta} + \pi _1 (E_2) \right) \subseteq
 \psi _2 ^{-1} ( \ker( \nu ))$ is immediate for the group homomorphism $\psi _2$ with $\ker (\psi _2) = E_2 (a \beta, b \alpha)$ and for  any
 $a \lambda _1 + b \lambda _2 + \Lambda (F, \mu) \in  \ker (\nu)$.
Towards the opposite inclusion
\begin{equation}     \label{NonTrivialInclusion}
\psi _2 ^{-1} ( \ker( \nu)) \subseteq \sum\limits _{a \lambda _1 + b \lambda _2 + \Lambda (F, \mu)} \left[ E_2 (a \beta, b \alpha) +
\left( \frac{\lambda _2}{\alpha} + \pi _1 (E_2), - \frac{\lambda _1}{\beta} + \pi _1 (E_2) \right) \right],
\end{equation}
let us pick up a point $(u + \pi _1 (E_2), v + \pi _1 (E_2)) \in \psi _2 ^{-1} ( \ker( \nu))$.
Then
\[
\psi _2 ( u + \pi _1 (E_2), v + \pi _1 (E_2)) = b \alpha u - a \beta v + \Lambda (F,\mu) =
 a \lambda _1 + b \lambda _2 + \Lambda (F,\mu) \in   \ker (\nu)
\]
for some $\lambda _1, \lambda _2 \in \pi _1 (E_1)$.
Straightforwardly,
\[
( u + \pi _1 (E_2), v + \pi _1 (E_2)) +
\left(  - \frac{\lambda _2}{\alpha} + \pi _1 (E_2), \frac{\lambda _1}{\beta} + \pi _1 (E_2) \right)
\in E_2(a \beta, b \alpha) = \ker( \psi _2),
\]
whereas
\[
( u + \pi _1(E_2), v + \pi _1 (E_2) ) \in E_2 (a \beta, b \alpha)  +
\left( \frac{\lambda _2}{\alpha}  + \pi _1 (E_2), - \frac{\lambda _1}{\beta} + \pi _1 (E_2) \right),
\]
which suffices for (\ref{NonTrivialInclusion}) and (\ref{PulledBackKernel}).
Altogether, we have proved that
\[
\mu ^{-1}(F) = \psi _2 ^{-1} ( \ker( \nu)) + (P_o,Q_o) =
\]
\[
= \sum\limits _{a \lambda _1 + b \lambda _2 + \Lambda (F, \mu)}  \left[ E_2(a \beta, b \alpha) + \left( \frac{p + \lambda _2}{\alpha} + \pi _1(E_2),
\frac{q - \lambda _1}{\beta} + \pi _1 (E_2) \right) \right].
\]

 \end{proof}

In order to specify  Lemma \ref{CoarseCuspLift}, let us recall that the integers ring $\mathcal{O}_{-d} = {\mathbb Z} + \omega _{-d} {\mathbb Z}$
of an imaginary quadratic number field ${\mathbb Q} ( \sqrt{-d})$ is a free ${\mathbb Z}$-module of rank $2$, generated by $1$ and
\[
\omega _{-d} = \begin{cases}
\sqrt{-d}   &  \mbox{ for $-d \not \equiv 1 {\rm mod}  \ \  4$,  }   \\
\frac{1 + \sqrt{-d}}{2}   &  \mbox{  for $-d \equiv 1 {\rm mod} \ \  4$.  }
\end{cases}
\]
An order $\mathcal{O}$ of a number field $K$ is a subring, which is a ${\mathbb Z}$-module with $\mathcal{O}  \otimes _{\mathbb Z} {\mathbb Q} = K$.
The orders of an imaginary quadratic number field ${\mathbb Q}( \sqrt{-d})$  are of the form
\[
\mathcal{O}_{-d,m} = {\mathbb Z} + m \mathcal{O}_{-d} = {\mathbb Z} + m \omega _{-d} {\mathbb Z}
\]
for some natural number $m = [ \mathcal{O}_{-d} : \mathcal{O}_{-d,m}]$, called the conductor of $\mathcal{O}_{-d,m}$.
All the orders $\mathcal{O}_{-d,m}$ are ${\mathbb Z}$-submodules of the maximal order $\mathcal{O}_{-d}$.

\begin{lemma}    \label{EndEqualPi1}
If the fundamental group of an elliptic curve  $E_{-d,m} = {\mathbb C} / \mathcal{O}_{-d,m}$     is an order $\mathcal{O}_{-d,m}$ of
 ${\mathbb Q}( \sqrt{-d})$, then the endomorphism ring $R_{-d,m} = End (E_{-d,m} ) = \mathcal{O}_{-d,m}$ coincides with the fundamental group.
\end{lemma}

\begin{proof}

The period ratio $m \omega _{-d}$ of $E_{-d,m}$ belongs to the imaginary quadratic number field ${\mathbb Q} ( \sqrt{-d})$,
 so that $E_{-d,m}$ has complex multiplication by ${\mathbb Q}( \sqrt{-d})$.
In other words, $R_{-d,m} = \mathcal{O}_{-d,c}$ is an order of ${\mathbb Q}( \sqrt{-d})$ with conductor $c \in {\mathbb N}$.
Note that $m \mathcal{O}_{-d}$ is a subring of $R_{-d,m}$, according to
$(m \mathcal{O}_{-d}) \mathcal{O}_{-d,m} \subseteq m \mathcal{O}_{-d} \subseteq \mathcal{O}_{-d,m}$.
Bearing in mind ${\mathbb Z} \subseteq R_{-d,m}$, one concludes that $\mathcal{O}_{-d,m} \subseteq R_{-d,m}   \subseteq \mathcal{O}_{-d}$.
Therefore the conductor $c = [ \mathcal{O}_{-d} : R_{-d,m}]$ of $R_{-d,m}$ divides the conductor $m = [ \mathcal{O}_{-d} : \mathcal{O}_{-d,m}]$ of $\mathcal{O}_{-d,m}$.

On the other hand, $1 \in \mathcal{O}_{-d,m}$ implies that $c \mathcal{O}_{-d} \subseteq R_{-d,m} \subseteq \mathcal{O}_{-d,m}$, whereas
$ R_{-d,m} \subseteq \mathcal{O}_{-d,m} \subseteq \mathcal{O}_{-d}$.
As a result, $m = [ \mathcal{O}_{-d} : \mathcal{O}_{-d,m}]$ divides $c = [ \mathcal{O}_{-d} : R_{-d,m}]$.

 The natural numbers  $m=c$ coincide, as far as divide each other.

\end{proof}

Note that     $\alpha \mathcal{O}_{-d,m} \subseteq \mathcal{O}_{-d} $ for all $\alpha \in \mathcal{O}_{-d}$.
 Therefore arbitrary $\alpha , \beta \in \mathcal{O}_{-d} \setminus \{ 0 \}$ provide an isogeny
 \[
 \mu = \left( \begin{array}{cc}
 \alpha  &  0   \\
 0  &  \beta
 \end{array}  \right) : E_{-d,m} \times E_{-d,m} \longrightarrow E_{-d} \times E_{-d}
 \]
 for $E_{-d} := E_{-d,1}$. The kernel
 \[
 \ker ( \mu) =
  \left( \frac{1}{\alpha} \mathcal{O}_{-d} / \mathcal{O}_{-d,m}  \right) \times \left( \frac{1}{\beta} \mathcal{O}_{-d} / \mathcal{O}_{-d,m}  \right)
\]
is of order $|\ker( \mu)| = [ \mathcal{O}_{-d} : \alpha \mathcal{O}_{-d,m} ] [ \mathcal{O}_{-d} : \beta \mathcal{O}_{-d,m}]$.
Making use of the inclusions
$
\alpha \mathcal{O}_{-d,m}    \subseteq   \alpha \mathcal{O}_{-d} \subseteq \mathcal{O}_{-d},
$
one computes that
\[
[ \mathcal{O}_{-d} : \alpha \mathcal{O}_{-d,m}] = [ \mathcal{O}_{-d} : \alpha \mathcal{O}_{-d}] [ \alpha \mathcal{O}_{-d} : \alpha \mathcal{O}_{-d,m} ] = |\alpha| ^2 m.
\]
Therefore $\mu$ is of degree $\deg ( \mu) = | \ker( \mu )| = m^2 | \alpha|^2 |\beta|^2$.

For an arbitrary algebraic integer  $\gamma \in \mathcal{O}_{-d} = {\mathbb Z} + \omega _{-d} {\mathbb Z}$ there exist uniquely determined
 $x( \gamma), y ( \gamma) \in {\mathbb Z}$ with $\gamma = x( \gamma) + \omega _{-d} y( \gamma)$.
One checks immediately that
  \[
  x: \mathcal{O}_{-d}  \longrightarrow {\mathbb Z} \ \ \mbox{  and}  \ \ y: \mathcal{O}_{-d} \longrightarrow {\mathbb Z}
  \]
  with $x( \gamma) + \omega _{-d} y( \gamma) = \gamma$ for $\forall \gamma \in \mathcal{O}_{-d}$
   are epimorphisms of $( \mathcal{O}_{-d}, +)$ on $({\mathbb Z}, +)$.

  We are going to specialize Lemma \ref{CoarseCuspLift} in the case of a unique factorization domain $\mathcal{O}_{-d}$, i.e., for
  \[
  d \in \{ 1, 2, 3, 7, 11, 19, 43, 67, 163 \}.
  \]
  If so, then for any $a_1, b_1 \in \mathcal{O}_{-d} \setminus \{ 0 \}$ there exists a greatest common divisor $GCD(a_1, b_1) \in \mathcal{O}_{-d}$ such that
  \[
  a = \frac{a_1}{GCD(a_1, b_1)}, \ \ b = \frac{b_1}{GCD(a_1, b_1)} \in \mathcal{O}_{-d} \setminus \{ 0 \}
  \]
   are relatively prime and
  $E_{-d} ( a_1, b_1) = E_{-d} ( a,b)$.
  Form now on, for ${\mathbb Q} ( \sqrt{-d})$ of class number $1$ and $E_{-d}  = ( {\mathbb C},+) / ( \mathcal{O}_{-d}, +)$, we label the elliptic curves
  $E_{-d} (a,b)$ through the origin of $E_{-d} \times E_{-d}$ by relatively prime $a, b \in \mathcal{O}_{-d}$.
  Then $a \mathcal{O}_{-d} + b \mathcal{O}_{-d} = GCD(a,b) \mathcal{O}_{-d} = \mathcal{O}_{-d}$ implies the existence of $a_o, b_o \in \mathcal{O}_{-d}$
  with $a a_o + b b_o =1$.
  For any $\gamma \in \mathcal{O}_{-d}$ the equality $a a_o + b b_o =1$ is equivalent to $a ( a_o + \gamma b) + b ( b_o - \gamma a) =1$,
   so that $a_o, b_o \in \mathcal{O}_{-d}$ with $aa_o + b b_o =1$ are not uniquely determined.

  \begin{corollary}     \label{cuspLift}
  Let   $\mathcal{O}_{-d}$ be the integers ring of an imaginary quadratic number field ${\mathbb Q}(\sqrt{-d})$ of class number $1$,
$
  x: ( \mathcal{O}_{-d}, +) \longrightarrow ({\mathbb Z},+), \ \
  y: ( \mathcal{O}_{-d}, +)   \longrightarrow ( {\mathbb Z}, +)
$
  be the epimorphisms, satisfying $x(\gamma) + \omega _{-d} y( \gamma) = \gamma$ for $\forall \gamma \in \mathcal{O}_{-d}$,
   $m$ be a natural number, $\mathcal{O}_{-d,m}  = {\mathbb Z} + m \mathcal{O}_{-d}$ be the order of ${\mathbb Q}( \sqrt{-d})$ with conductor $m$,
   $E_{-d,m} = ( {\mathbb C}, +) / ( \mathcal{O}_{-d,m}, +)$ be the elliptic curve with fundamental group $\mathcal{O}_{-d,m}$ and
   $\alpha , \beta \in \mathcal{O}_{-d} \setminus \{ 0 \}$.
  For any  relatively prime $(a,b) \in \mathcal{O}_{-d} ^2 \setminus \{ (0,0) \}$ consider the elliptic curve $F = E_{-d} (a,b) + (P,Q)$ with
   $P = p + \mathcal{O}_{-d}$,  $Q = q + \mathcal{O}_{-d} \in E_{-d}$ for some $p,q \in {\mathbb C}$, the lattice
    $\Lambda ( F, \mu) = a \beta \mathcal{O}_{-d,m} + b \alpha \mathcal{O}_{-d,m}$ and some
  $a_o, b_o \in \mathcal{O}_{-d}$ with $a a_o + b b_o =1$.
  Then the isogeny
  \[
  \mu = \left( \begin{array}{cc}
  \alpha  &  0   \\
  0  &  \beta
  \end{array}  \right) : E_{-d,m} \times E_{-d,m}   \longrightarrow E_{-d} \times E_{-d}
  \]
  of degree $\deg ( \mu) = m^2 | \alpha|^2 |\beta|^2$ pulls back $F \subset E_{-d} \times E_{-d}$ to a disjoint union


  \begin{equation}   \label{ExplicitPreimage}
  \mu ^{-1}(F) = \sum\limits _{\lambda + \Lambda (F, \mu)}  \left[ E_m (a \beta, b \alpha) +
  \left( \frac{ p + b_o \lambda}{\alpha}  + \Lambda _m, \frac{q - a_o \lambda}{\beta}  + \Lambda _m \right) \right]
  \end{equation}
of $[ \mathcal{O}_{-d} : \Lambda (F, \mu)]$ mutually parallel smooth irreducible elliptic  curves.

Let
\[
\delta := GCD(a \beta, b \alpha),  \ \ \xi := \frac{a \beta}{\delta}, \ \  \eta := \frac{b \alpha}{\delta} \in \mathcal{O}_{-d}.
\]
In the case of $(x( \xi), x( \eta)) \neq (0,0)$ introduce $x_o := GCD(x( \xi), x( \eta)) \in {\mathbb N}$ and  define
$y_o := GCD(y( \xi), y(\eta))$ whenever $(y( \xi), y( \eta)) \neq (0,0)$.
Then
\[
[ \mathcal{O}_{-d} : \Lambda (F, \mu)] =
\]
\[
= \begin{cases}
|\delta|^2 GCD(x_o,m)GCD(y_o,m) &  \mbox{ for $(x(\xi), x(\eta)) \neq (0,0)$, $(y(\xi), y(\eta)) \neq (0,0)$,  }  \\
m | \delta |^2 GCD(x_o,m)   &  \mbox{   for  $(x(\xi),x(\eta)) \neq (0,0)$, $y(\xi) = y(\eta) =0$, }  \\
m |\delta|^2 GCD(y_o,m)  &  \mbox{ for  $x(\xi) = x(\eta) =0$, $(y(\xi), y(\eta)) \neq (0,0)$. }
\end{cases}
\]

  \end{corollary}

  \begin{proof}

For $(a,b) \in \mathcal{O}_{-d}^2 \setminus  \{ (0,0) \}$ with $GCD(a,b)=1$ the lattice
\[
\Lambda (F) := a \pi _1 (E_1) + b \pi _1 (E_1) = a \mathcal{O}_{-d} + b \mathcal{O}_{-d} = \mathcal{O}_{-d}.
\]
Formula (\ref{ExplicitPreimage}) is an  immediate consequence of (\ref{ExplicitCoarsePreimage}) with
\[
\lambda = 1. \lambda = (a a_o + bb_o) \lambda = a( a_o \lambda) + b ( b_o \lambda) \ \ \mbox{  for } \ \   \lambda \in \mathcal{O}_{-d}, \ \
\lambda _1 = a_o \lambda, \ \ \lambda _2 = b_o \lambda.
\]

Towards the explicit calculation of the index $[ \mathcal{O}_{-d} : \Lambda (F, \mu)]$, one makes use of the inclusion
 $m \mathcal{O}_{-d} \subseteq \mathcal{O}_{-d,m}$  and  observes that
\[
\Lambda (F, \mu) = a \beta  \mathcal{O}_{-d,m}  + b \alpha \mathcal{O}_{-d,m}  \supseteq a \beta m \mathcal{O}_{-d} + b \alpha m \mathcal{O}_{-d} =
m (a \beta \mathcal{O}_{-d} + b \alpha \mathcal{O}_{-d}) =
\]
\[
= m GCD(a \beta, b \alpha) \mathcal{O}_{-d} = m \delta \mathcal{O}_{-d}.
\]
The sequence of inclusions
$
m \delta \mathcal{O}_{-d} \subseteq \Lambda (F, \mu) \subseteq \mathcal{O}_{-d}
$
provides
\begin{equation}    \label{FirstReduction}
[ \mathcal{O}_{-d} : \Lambda (F, \mu)] = \frac{[\mathcal{O}_{-d} : m \delta \mathcal{O}_{-d}]}{[ \Lambda (F, \mu) : m \delta \mathcal{O}_{-d}]} =
\frac{m^2 | \delta|^2}{[ \Lambda (F, \mu) : m \delta \mathcal{O}_{-d} ]}.
\end{equation}
The algebraic integers  $\xi = \frac{a \beta}{\delta}$, $\eta = \frac{b \alpha}{\delta} \in \mathcal{O}_{-d}$ are relatively prime in $\mathcal{O}_{-d}$ and the lattice
\[
\delta ^{-1} \Lambda (F, \mu) = \xi \mathcal{O}_{-d,m} + \eta \mathcal{O}_{-d,m} =
\]
\[
= \xi ( {\mathbb Z} + m \omega _{-d} {\mathbb Z}) + \eta ( {\mathbb Z} + m \omega _{-d} {\mathbb Z}) =
 ( \xi {\mathbb Z} + \eta {\mathbb Z}) + m \omega _{-d} ( \xi {\mathbb Z} + \eta {\mathbb Z})
 \]
 contains $m \mathcal{O}_{-d} = m GCD(\xi, \eta) \mathcal{O}_{-d} = \xi ( m \mathcal{O}_{-d}) + \eta ( m \mathcal{O}_{-d})$ by
  $m \mathcal{O}_{-d} \subseteq \mathcal{O}_{-d,m}$.
 Therefore the quotient
 \[
 \delta  ^{-1} \Lambda (F, \mu) / m \mathcal{O}_{-d} =
     [ ( \xi {\mathbb Z} + \eta {\mathbb Z}) + m \omega_{-d} (\xi {\mathbb Z} + \eta {\mathbb Z})] / m \mathcal{O}_{-d} \supseteq
     \]
     \[
\supseteq      [ ( \xi {\mathbb Z} + \eta {\mathbb Z}) + m \mathcal{O}_{-d}] / m \mathcal{O}_{-d}.
      \]
On the other hand, $m \omega _{-d} ( \xi {\mathbb Z} + \eta {\mathbb Z}) \subseteq m \mathcal{O}_{-d}$ implies the opposite inclusion
\[
\delta ^{-1} \Lambda (F, \mu) / m \mathcal{O}_{-d} =
[ ( \xi {\mathbb Z} + \eta {\mathbb Z}) + m \omega _{-d} ( \xi {\mathbb Z} + \eta {\mathbb Z})] / m \mathcal{O}_{-d} \subseteq
\]
\[
\subseteq [ ( \xi {\mathbb Z} + \eta {\mathbb Z}) + m \mathcal{O}_{-d}] / m \mathcal{O}_{-d},
\]
whereas the coincidence
\[
\delta ^{-1} \Lambda (F, \mu) / m \mathcal{O}_{-d} = [ ( \xi {\mathbb Z} + \eta {\mathbb Z}) + m \mathcal{O}_{-d}] / m \mathcal{O}_{-d} \simeq
( \xi {\mathbb Z} + \eta {\mathbb Z}) / [ ( \xi {\mathbb Z} + \eta {\mathbb Z}) \cap m \mathcal{O}_{-d}].
\]
As a result,
\begin{equation}   \label{SecondReduction}
[ \delta ^{-1} \Lambda (F, \mu) : m \mathcal{O}_{-d}] =
 [ ( \xi {\mathbb Z} + \eta {\mathbb Z} ) : ( \xi {\mathbb Z} + \eta {\mathbb Z}) \cap m \mathcal{O}_{-d}].
\end{equation}
Towards the calculation of the last index, consider the group epimorphism
\[
f: ( \mathcal{O}_{-d},+)  \longrightarrow ( {\mathbb Z}_m \times {\mathbb Z}_m, +),
\]
\[
f( \gamma) = ( x( \gamma) + m {\mathbb Z}, \ \  y( \gamma) + m {\mathbb Z}) \ \ \mbox{  for  } \ \ \forall \gamma \in \mathcal{O}_{-d}.
\]
Its kernel is
$
\ker(f) = m {\mathbb Z} + m \omega _{-d} {\mathbb Z} = m ( {\mathbb Z} +  \omega _{-d} {\mathbb Z}) = m \mathcal{O}_{-d}.
$
Let us denote
\[
x_1 := x(\xi), \ \ y_1 := y(\xi), \ \ x_2 := x(\eta), \ \ y_2 := y(\eta) \in {\mathbb Z}
\]
and observe that the lattice
\[
\xi {\mathbb Z} + \eta {\mathbb Z} = \{ (x_1 + y_1 \omega _{-d} ) z_1 + ( x_2 + y_2 \omega _{-d} ) z_2 \ \ \vert \ \  z_1, z_2 \in {\mathbb Z} \},
\]
\begin{equation}   \label{ThirdReduction}
\xi {\mathbb Z} +\eta {\mathbb Z} = \{ (x_1 z_1 + x_2 z_2) + \omega _{-d} ( y_1 z_1 + y_2 z_2) \ \ \vert \ \  z_1, z_2 \in {\mathbb Z} \}.
\end{equation}
Assume that $(x_1, x_2) \neq (0,0)$, $(y_1,y_2) \neq (0,0)$ and put $x_o := GCD(x_1, x_2) \in {\mathbb N}$, $y_o := GCD(y_1, y_2) \in {\mathbb N}$.
Then
\[
x'_1 := \frac{x_1}{x_o}, \ \  x'_2 := \frac{x_2}{x_o} \in {\mathbb Z}
\]
 are relatively prime, as well as
\[
y'_1 := \frac{y_1}{y_o}, \ \  y'_2 := \frac{y_2}{y_o} \in {\mathbb Z}.
\]
That allows to describe the lattice $\xi {\mathbb Z} + \eta {\mathbb Z}$ in the form
\begin{equation}   \label{LatticeDescription}
\xi {\mathbb Z} + \eta {\mathbb Z} = \{  x_o (x'_1 z_1 + x'_2 z_1) + \omega _{-d} y_o ( y'_1 z_1 + y'_2 z_2) \ \ \vert \ \  z_1, z_2 \in {\mathbb Z} \}.
\end{equation}
According to $x'_1 {\mathbb Z} + x'_2 {\mathbb Z} = GCD(x'_1, x'_2) {\mathbb Z} = {\mathbb Z}$, the image of $x: ( \xi {\mathbb Z} + \eta {\mathbb Z}, +) \rightarrow ({\mathbb Z},+)$ is the free ${\mathbb Z}$-module $x_o {\mathbb Z}$.
Similarly, $y'_1 {\mathbb Z} + y'_2 {\mathbb Z} = GCD(y'_1, y'_2) ) {\mathbb Z} = {\mathbb Z}$  reveals that  the image of $y: ( \xi {\mathbb Z} + \eta {\mathbb Z}, +) \rightarrow ({\mathbb Z},+)$ is $ y_o {\mathbb Z}$.
 Therefore the group homomorphism
 \[
 f: ( \xi {\mathbb Z} + \eta {\mathbb Z}, +) \longrightarrow ( {\mathbb Z}_m \times {\mathbb Z}_m, +)
 \]
 has image
 \[
 f ( \xi {\mathbb Z} + \eta {\mathbb Z}) =
  \left[ ( x_o {\mathbb Z} + m {\mathbb Z}) / m {\mathbb Z} \right] \times  \left[  ( y_o {\mathbb Z} + m {\mathbb Z} ) / m {\mathbb Z} \right] =
  \]
  \[
= [ GCD(x_o,m) {\mathbb Z} / m {\mathbb Z} ] \times [ GCD(y_o, m) {\mathbb Z} / m {\mathbb Z}] =
{\mathbb Z} _{ \frac{m}{GCD(x_o,m)}} \times {\mathbb Z} _{ \frac{m}{GCD(y_o,m)}}
 \]
 of cardinality
 \[
 | f( \xi {\mathbb Z} + \eta {\mathbb Z} ) | = \frac{m^2}{GCD(x_o,m) GCD(y_o,m)}.
 \]
 By the Isomorphism Theorem
 \[
 ( \xi {\mathbb Z} + \eta {\mathbb Z}) / [ ( \xi {\mathbb Z} + \eta {\mathbb Z} ) \cap m \mathcal{O}_{-d}] =
 ( \xi {\mathbb Z} + \eta {\mathbb Z}) / [ ( \xi {\mathbb Z} + \eta {\mathbb Z}) \cap \ker (f) ] \simeq f ( \xi {\mathbb Z} + \eta {\mathbb Z}),
 \]
 one concludes that
\[
[ \delta ^{-1} \Lambda (F, \mu) : m \mathcal{O}_{-d}] =
 [ ( \xi {\mathbb Z} + \eta {\mathbb Z}) : ( \xi {\mathbb Z} + \eta {\mathbb Z}) \cap m \mathcal{O}_{-d} ] =
 \]
 \[
= | f( \xi {\mathbb Z} + \eta {\mathbb Z}) | = \frac{m^2}{GCD(x_o, m) GCD(y_o,m)}.
\]
Putting together with (\ref{FirstReduction}), one obtains that
\[
[ \mathcal{O}_{-d} : \Lambda (F, \mu ) ] = | \delta | ^2 GCD(x_o, m) GCD(y_o,m)
\]
for $(x_1,x_2) \neq (0,0)$, $(y_1, y_2) \neq (0,0)$.

If $(y_1,y_2) = (0,0)$ then $(x_1,x_2) \neq (0,0)$, as far as the pair  $\left( \frac{a \beta}{\delta}, \frac{b \alpha}{\delta} \right) = ( \xi, \eta) =
( x_1 + y_1 \omega _{-d}, x_2 + y_2 \omega _{-d}) \neq (0,0)$.
By (\ref{ThirdReduction}) one has
$
\xi {\mathbb Z} + \eta {\mathbb Z} = x_1 {\mathbb Z} + x_2 {\mathbb Z} = GCD(x_1, x_2) {\mathbb Z} = x_o {\mathbb Z}.
$
Note that
\[
( \xi {\mathbb Z} + \eta {\mathbb Z}) \cap m  \mathcal{O}_{-d} =
x_o {\mathbb Z} \cap ( m {\mathbb Z} + m \omega _{-d} {\mathbb Z}) =
\]
\[
= x_o {\mathbb Z} \cap m {\mathbb Z} =
 LCM(x_o,m) {\mathbb Z} =
 \frac{x_o m}{GCD(x_o,m)} {\mathbb Z}
\]
for the least common multiple $LCM(x_o,m)$ of $x_o \in {\mathbb Z}  \setminus \{ 0 \}$ and $m \in {\mathbb N}$.
As a result, (\ref{SecondReduction}) reads as
\[
[ \delta ^{-1} \Lambda (F, \mu) : m \mathcal{O}_{-d}] = \left[  x_o {\mathbb Z} : \frac{x_o m}{GCD(x_o,m)} {\mathbb Z} \right] =
\frac{m}{GCD(x_o,m)}
\]
and (\ref{FirstReduction}) provides
\[
[ \mathcal{O}_{-d} : \Lambda (F, \mu) ] = m | \delta| ^2 GCD(x_o,m) \ \ \mbox{   for  } \ \  (x_1,x_2) \neq (0,0), \ \ y_1 = y_2 =0.
\]

In a similar vein, for $(x_1,x_2) = (0,0)$ there follows $(y_1,y_2) \neq (0,0)$.
According to (\ref{ThirdReduction}),
$
\xi {\mathbb Z} + \eta {\mathbb Z} = \omega _{-d} ( y_1 {\mathbb Z} + y_2 {\mathbb Z}) = \omega _{-d} GCD(y_1,y_2) {\mathbb Z} =  \omega _{-d} y_o{\mathbb Z}.
$
Therefore
\[
( \xi {\mathbb Z} + \eta {\mathbb Z}) \cap m \mathcal{O}_{-d} = \omega _{-d} y_o  {\mathbb Z} \cap ( m {\mathbb Z} + m \omega _{-d} {\mathbb Z}) =
 \omega _{-d} y_o {\mathbb Z}  \cap \omega _{-d}  m {\mathbb Z} = \omega _{-d} ( y_o {\mathbb Z} \cap m {\mathbb Z}) =
\]
\[
= \omega _{-d} LCM(y_o,m) {\mathbb Z} = \frac{y_o  m}{GCD(y_o,m)}  \omega _{-d} {\mathbb Z}
\]
and (\ref{SecondReduction}) provides
\[
[ \delta ^{-1} \Lambda (F, \mu) : m \mathcal{O}_{-d}] = \left[  \omega _{-d} y_o {\mathbb Z} : \frac{y_o m}{GCD(y_0,m)} \omega _{-d} {\mathbb Z} \right] =
\frac{m}{GCD(y_o,m)}.
\]
By (\ref{FirstReduction}), one has
\[
[ \mathcal{O}_{-d} : \Lambda (F, \mu) ] = m | \delta|^2 GCD(y_o,m) \ \ \mbox{   for }  \ \  x_1 = x_2 =0, \ \ (y_1,y_2) \neq (0,0).
\]

  \end{proof}

The immediate application of Corollary \ref{cuspLift} to Hirzebruch's proportional elliptic configuration
\[
D_{-3} ^{\rm Hir} = E_{-3} (1,0) + E_{-3} (0,1) + E_{-3} (1,1) + E_{-3} \left( 1, e^{\frac{\pi i}{3}} \right)   \subset A_{-3}
\]
from \cite{Hir84}  yields the following

\begin{corollary}     \label{HirzebruchsPullBack}
Let us consider the ring  $\mathcal{O}_{-3}$ of Eisenstein integers, the  elliptic curve
$E_{-3,m} = {\mathbb C} / \mathcal{O}_{-d,m}$,  whose fundamental group
$\mathcal{O}_{-d,m} = {\mathbb Z} + m \mathcal{O}_{-3}$ is the order $\mathcal{O}_{-d,m}$ of ${\mathbb Q}( \sqrt{-3})$ with conductor $m \in {\mathbb N}$,
 $E_{-3} := E_{-3,1}$   and  the isogeny
\[
\mu = \left( \begin{array}{cc}
\alpha  &  0  \\
0  &  1
\end{array}  \right) : E_{-d,m} \times E_{-d,m} \longrightarrow E_{-3} \times E_{-3}
\]
 of the corresponding Cartesian squares.
 Denote by
 \[
 n ( \mu ^{-1} (F)) = [ \mathcal{O}_{-3} : a \mathcal{O}_{-d,m} + b \alpha \mathcal{O}_{-d,m}]
 \]
  the number  of the irreducible components of the pull-back $\mu ^{-1} (F)$ of an elliptic curve $F = E_{-3}(a,b) + (P,Q) \subset E_{-3} \times E_{-3}$ by $\mu$  and put $x: ( \mathcal{O}_{-3}, +) \rightarrow ({\mathbb Z}, +)$, $y: ( \mathcal{O}_{-3}, +) \rightarrow ({\mathbb Z},+)$ for  the epimorphisms with
 $x( \alpha) + \omega _{-3} y( \alpha) = \alpha$ for all $\alpha \in \mathcal{O}_{-3}$, $\omega _{-3} := e^{\frac{ \pi i}{3}}$.
  Then $\mu$ is   of degree $\deg ( \mu) = m^2 | \alpha|^2$ and
\[
n ( \mu ^{-1} E_{-3}(1,0)) = m, \ \ n ( \mu ^{-1} E_{-3} (0,1)) = m | \alpha | ^2,
\]
\[
n ( \mu ^{-1} E_{-3} (1,1) )= GCD(y( \alpha), m), \ \ n \left( \mu ^{-1} E_{-3} \left( 1, e^{\frac{ \pi i}{3}} \right) \right) =
 GCD(x( \alpha) + y( \alpha), m).
\]

 In particular, the proportional elliptic configuration $\mu ^{-1} D_{-3} ^{\rm Hir} \subset E_{-d,m} \times E_{-d,m}$ has
 \[
 h ( \mu ^{-1} D_{-3} ^{\rm Hir}) = m + m | \alpha|^2 + GCD(y( \alpha), m) + GCD(x( \alpha ) + y( \alpha), m)
 \]
 smooth elliptic irreducible components.
\end{corollary}

For $m=1$ and $E_{-3,1} = E_{-3}$ one has $\deg ( \mu) = | \alpha|^2$ and $h( \mu ^{-1} D_{-3} ^{\rm Hir}) = | \alpha|^2 + 3$ cusps of the non-compact  torsion free ball quotients ${\mathbb B} / \Gamma _{\mu}$, associated with $\mu  ^{-1} D_{-3} ^{\rm Hir}$.
Note that the irreducible components of $D_{-3}^{\rm Hir}$, different from $E_{-3} (0,1)$ pull back to irreducible smooth elliptic curves and only the number of the irreducible components of $\mu ^{-1} E_{-3}(0,1)$ increases with $|\alpha|^2$.
In order to obtain  a proportional elliptic con\-fi\-gu\-ra\-tion $D \subset A_{-3}$, whose pull backs $\mu ^{-1} D \subset A_{-3}$ have fixed number of irreducible components and infinitely increasing number of singular points $|\alpha|^2 \to \infty$, one maps isomorphically $D_{-3} ^{\rm Hir}$ into a proportional elliptic configuration $D_{-3}^{(1,4)} \subset A_{-3}$ without irreducible components, parallel to $E_{-3}(0,1)$.
More precisely, the linear transformation
\[
g = \left( \begin{array}{cc}
1  &  e^{\frac{\pi i}{3}}  \\
0  &  1
\end{array}  \right) \in GL(2, \mathcal{O}_{-3}) < Aut (A_{-3})
\]
maps
$
D_{-3} ^{\rm Hir} = E_{-3} (1,0) + E_{-3} (0,1) + E_{-3} (1,1) + E_{-3} \left( 1, e^{\frac{\pi i}{3}} \right)
$
onto the proportional elliptic configuration
\begin{equation}   \label{ModiifedHirzebruchsPEC}
D_{-3} ^{(1,4)} = E_{-3}(1,0) + E_{-3} \left( e^{\frac{\pi i}{3}} , 1 \right) + E_{-3} \left( \sqrt{-3} e^{ - \frac{\pi i}{3}}, 1 \right) + E_{-3} (1,1).
\end{equation}
In analogy with Corollary \ref{HirzebruchsPullBack} one has

\begin{corollary}   \label{ModifiedHirzebruchsPullBack}
Let us consider the ring  $\mathcal{O}_{-3}$   of Eisenstein integers, the elliptic curve
$E_m = {\mathbb C} / \mathcal{O}_{-3,m} $,  whose fundamental group is the order $\mathcal{O}_{-3,m} = {\mathbb Z} + m \mathcal{O}_{-3}$ of ${\mathbb Q}( \sqrt{-3})$ with conductor $m \in {\mathbb N}$, $E_{-3} := E_{-3,1}$,  the isogeny
\[
\mu = \left( \begin{array}{cc}
\alpha  &  0  \\
0  &  1
\end{array}   \right) : E_{-3,m} \times E_{-3,m} \longrightarrow E_{-3} \times E_{-3},
\]
and the epimorphisms $x: ( \mathcal{O}_{-3}, +)    \rightarrow ({\mathbb Z},+)$, $y: ( \mathcal{O}_{-3}, +)   \rightarrow ({\mathbb Z},+)$ of additive groups  with $x( \alpha) + \omega _{-3} y( \alpha) = \alpha$ for $\forall \alpha \in \mathcal{O}_{-3}$, $\omega _{-3} := e^{\frac{\pi i}{3}}$.
   Denote by
   \[
   n ( \mu ^{-1} (F))  = [ \mathcal{O}_{-3} : a \mathcal{O}_{-3,m}  + b \alpha \mathcal{O}_{-3,m}]
   \]
    the number of the irreducible components of the pull back
 $\mu ^{-1} (F)$ of an elliptic curve $F = E_{-3}(a,b) + (P,Q) \subset E_{-3} \times E_{-3}$ by $\mu$.
 Then $\deg ( \mu ) = m |\alpha|^2$,
 \[
 n \left( \mu ^{-1} E_{-3} \left( e^{\frac{\pi i}{3}}, 1 \right) \right) = GCD( x( \alpha), m),
 \]
 \[
n \left(  \mu ^{-1} E_{-3} \left( \sqrt{-3} e^{- \frac{\pi i}{3}}, 1 \right) \right) =
\begin{cases}
1  &  \text{  for $\alpha \not \in \sqrt{-3} \mathcal{O}_{-3}$, }  \\
3  &  \text{ for $\alpha \in \sqrt{-3} \mathcal{O}_{-3}$,  }
\end{cases}
\]
and  the number of the irreducible components of the proportional elliptic configuration $\mu ^{-1} D_{-3} ^{(1,4)} \subset E_{-d,m} \times E_{-d,m}$  is
\[
h \left(  \mu D_{-3} ^{(1,4)} \right) =
\begin{cases}
m + GCD(x( \alpha),m) + GCD(y( \alpha), m) + 1 &  \text{ for $\alpha \not \in \sqrt{-3} \mathcal{O}_{-3}$,  }  \\
m + GCD(x( \alpha), m) + GCD(y(\alpha), m) + 3  &  \text{  for $\alpha \in \sqrt{-3} \mathcal{O}_{-3}$}
\end{cases}
\]

In particular, if $m=1$ then
\[
n ( \mu ^{-1} D_{-3} ^{(1,4)} ) =
\begin{cases}
4  &  \text{  for $\alpha \not \in \sqrt{-3} \mathcal{O}_{-3}$,  }  \\
6  & \text{  for $\alpha \in \sqrt{-3} \mathcal{O}_{-3}$. }
\end{cases}
\]
\end{corollary}

\section{Isogeny series of co-abelian torsion free ball quotients with infinitely increasing volumes }

The previous considerations provide infinite isogeny series $( {\mathbb B} / \Gamma _n)'$ of torsion free toroidal compactifications with infinitely increasing volumes.

\begin{corollary}     \label{EisensteinIsogenySeriesWithBirationalTerms}
For an arbitrary sequence $\{ \gamma _n \} _{n=1} ^{\infty} \subset \mathcal{O}_{-3} \setminus ( \mathcal{O}_{-3}^* \cup \{ 0 \})$,  consider the sequence $\{ \alpha _n = \prod\limits _{j=1} ^n \gamma _j \} _{n=1} ^{\infty} \subset \mathcal{O}_{-3}$ and the isogenies
\[
\lambda _n = \left(  \begin{array}{cc}
\gamma_n  &  0   \\
0  &   -1
\end{array}  \right)  :  E_{-3} \times E_{-3}  \longrightarrow E_{-3} \times  E_{-3} \ \ \mbox{  for  } \ \ \forall n \in {\mathbb N}.
\]

(i) Then
\[
\begin{diagram}
\node{( {\mathbb B} / \Gamma _{-3} ^{\rm Hir})'}   \node{( {\mathbb B} / \Gamma _1)'}  \arrow{w,t}{\lambda '_1}
\node{( {\mathbb B} / \Gamma _2)'}  \arrow{w,t}{\lambda '_2}  \node{\mbox{  }}  \arrow{w,t}{\lambda '_3}          \node{\ldots }
\end{diagram}
\]
is an infinite isogeny sequence  of torsion free, Picard modular  toroidal compactifications over ${\mathbb Q}( \sqrt{-3})$, 
birational to $E_{-3} \times E_{-3}$,
with infinitely increasing volume
\[
{\rm vol} ( {\mathbb B} / \Gamma _n) = \frac{8 \pi ^2}{3} | \alpha _n |^2 = \frac{8 \pi ^2}{3} \prod\limits _{j=1} ^n | \gamma _j |^2
\]
and infinitely increasing number of cusps
\[
h ( {\mathbb B} / \Gamma _n) = | \alpha _n|^2 + 3 = \prod\limits _{j=1} ^n  | \gamma _j|^2 +3.
\]

(ii) If $\{ \gamma _n \} _{n=1} ^{\infty} \subset \mathcal{O}_{-3} \setminus ( \mathcal{O}_{-3} \cup \sqrt{-3} \mathcal{O}_{-3})$ then
\[
\begin{diagram}
\node{( {\mathbb B} / \Gamma _{-3} ^{(1,4)})'}   \node{( {\mathbb B} / \Gamma _1^{(1,4)})'}  \arrow{w,t}{\lambda '_1}
\node{( {\mathbb B} / \Gamma _2 ^{(1,4)})'}  \arrow{w,t}{\lambda '_2}  \node{\mbox{  }}  \arrow{w,t}{\lambda '_3}          \node{\ldots }
\end{diagram}
\]
is an infinite isogeny series of torsion free, Picard modular toroidal compactifications over ${\mathbb Q}(\sqrt{-3})$, birational to $E_{-3} \times E_{-3}$, with  four cusps  and infinitely increasing volume
\[
{\rm vol} ( {\mathbb B} / \Gamma _n ^{(1,4)}) = \frac{8 \pi ^2}{3} | \alpha _n|^2 = \frac{ 8 \pi ^2}{3} \prod\limits _{j=1} ^n | \gamma _j |^2.
\]
\end{corollary}

 Corollary  \ref{EisensteinIsogenySeriesWithBirationalTerms} follows from Corollaries \ref{HirzebruchsPullBack} and \ref{ModifiedHirzebruchsPullBack} with $\mu _n = \lambda _n \ldots \lambda _2 \lambda _1$.

In order to construct isogeny series with infinitely increasing volumes and non-birational terms, we proceed with a characterization of the birational Cartesian squares of elliptic curves.

\begin{lemma}   \label{Birational1}
Let $E_1$ and $E_2$ be elliptic curves.
Then any birational map
\[
f: E_2 \times E_2 \longrightarrow E_1 \times E_1
\]
 is biregular.
\end{lemma}

\begin{proof}

Let us denote $A_j = E_j \times E_j$ for $1 \leq j \leq 2$ and put  by $\mathcal{D}_f$ the non-empty Zariski open subset of $A_2$, on which $f$ is defined and biregular.
 In other words, $\mathcal{D}_f$ is the intersection of the regularity domain of $f$ with the image of the regularity domain of $f^{-1}$ under $f^{-1}$.
 Assume that $\mathcal{D}_f \neq A_2$.
 The abelian surface $A_2 = \cup _{P \in E_2} P \times E_2$ foliates by elliptic curves $P \times E_2$, isomorphic to $E_2$.
  The Zariski closed subset $A_2 \setminus \mathcal{D}_f$ contains at most finitely many $P_1 \times E_2, \ldots , P_k \times E_2 $.
  For any $P \in E_2 \setminus \{ P_1, \ldots , P_k \}$ the  biregular restriction
   $f: (P \times E_2 ) \cap \mathcal{D}_f \rightarrow f ( ( P \times E_2 ) \cap \mathcal{D} _f )$ can be viewed as a birational map
   $f: ( P \times E_2) \myarrow \overline{f ( ( P \times E_2) \cap \mathcal{D}_f )}$ in the Zariski closures
   $\overline{f(( P \times E_2) \cap \mathcal{D}_f )}$ of $f(( P \times E_2 ) \cap \mathcal{D}_f)$ in $A_1$.
   The elliptic curves $P \times E_2$ are smooth, so that
\[
f: P \times E_2 \longrightarrow \overline{ f (( P \times E_2) \cap \mathcal{D}_f )}
\]
are biregular for $\forall P \in E_2 \setminus \{ P_1, \ldots , P_k \}$.
Thus, $\cup _{P \in E_2 \setminus \{ P_1, \ldots , P_k \}} (P \times E_2) \subseteq \mathcal{D}_f$. In order to justify the biregularity  of $f$ on $A_2$,
 let us consider the other natural  foliation $A_2 = \cup _{Q \in E_2} E_2 \times Q$ of $A_2$.
The restrictions $f: E_2 \times Q \rightarrow f( E_2 \times Q)$ are biregular for all but at most finitely many $Q \in E_2 \setminus \{ Q_1, \ldots , Q_l \}$.
 As a result, $\cup _{Q \in E_2 \setminus \{ Q_1, \ldots , Q_l \}} (E_2 \times Q) \subseteq \mathcal{D}_f$ and
\[
 A_2 \setminus \mathcal{D}_f \subseteq \left( \cup _{i=1} ^k ( P_i \times E_2) \right) \cap \left( \cup _{j=1} ^l ( E_2 \times Q_j ) \right) = \{ (P_i, Q_j) \ \ \vert \ \  1 \leq i \leq k, \ \  1 \leq j \leq l \}
\]
is at most a finite set of points.
Now,
\[
\left[ \cup _{i=1} ^k P_i \times  (E_2 \setminus \{ Q_1, \ldots , Q_l \} ) \right] \cup \left[ \cup _{j=1} ^l (E_2 \setminus \{ P_1, \ldots , P_k \} ) \times Q_j \right] \subset \mathcal{D}_f,
\]
contrary to the choice of
$\left[ \cup _{i=1} ^k P_i \times E_2 \right] \cup \left[ \cup _{j=1} ^l E_2 \times Q_j \right] \subseteq (A_2 \setminus   \mathcal{D}_f )$.
 The contradiction justifies that $A_2 = \mathcal{D}_f$ and $f: A_2 \rightarrow f(A_2)$ is a biregular map.
  The Cartesian square  $A_2 = E_2 \times E_2$ of the elliptic curve $E_2$  is a projective variety, so that the image
   $f(A_2)$ of $f : A_2 \rightarrow A_1$ is  Zariski closed in $A_1$.
  On the other hand, the morphism $f: A_2 \rightarrow A_1$ of abelian varieties is a group homomorphism, after an appropriate choice of an origin
  $\check{o} _{A_1}$ of $A_1$.
   The only Zariski dense abelian subvariety of $(A_1, +)$ is $A_1$ itself, so that $f(A_2) = A_1$ and $f: A_2 \rightarrow A_1$ is biregular.

\end{proof}

\begin{lemma}   \label{R1NotR2NotBiregular}
Let $E_j$, $1 \leq j \leq 2$ be elliptic curves with different endomorphism rinds $End(E_1) = R_1 \neq R_2=End(E_2)$.
Then the abelian surfaces $A_1 = E_1 \times E_1$ and $A_2 = E_2 \times E_2$ are not birational.

In particular, if   $\mathcal{O}_{-d,m}$, $\mathcal{O}_{-d,n}$ are   orders of an imaginary quadratic number field  ${\mathbb Q}( \sqrt{-d})$
  with different conductors $m,n \in {\mathbb N}$  and $E_{-d,m} := {\mathbb C} / \mathcal{O}_{-d,m}$, $E_{-d,n} = {\mathbb C} / \mathcal{O}_{-d,n}$,
  then the abelian surfaces $E_{-d,m} \times E_{-d,m}$ and $E_{-d,n} \times E_{-d,n}$ are not birational.
\end{lemma}

\begin{proof}

Assume the opposite and consider a birational map $f: A_2 \rightarrow A_1$.
By Lemma \ref{Birational1}, $f$ is biregular.
After moving the origin $\check{o}_{A_1}$ of $A_1$ at $f( \check{o}_{A_2})$, the isomorphism $f: (A_2, +) \rightarrow (A_1,+)$ is a group homomorphism and
\[
f = \left( \begin{array}{cc}
a  &  b  \\
c  &  d
\end{array}  \right) \in GL(2, {\mathbb C})
\]
is represented by a non-singular matrix.
The elliptic curve $E_2 \times \check{o}_{E_2}$ is a subgroup of $(A_2, +)$.
The homomorphism $f: (A_2, +) \rightarrow (A_1, +)$ maps it isomorphically onto an elliptic curve $f( E_2 \times \check{o}_{E_2}) = E_1 (a,b) \subset A_1$ through the origin $\check{o}_{A_1}$ with slope vector $(a,b) \in {\mathbb C}^2 \setminus \{ (0,0) \}$.
The induced map $f_* : \pi _1 ( E_2 \times \check{o}_{E_2}) \rightarrow \pi _1 (E_1(a,b))$ of the fundamental groups is a group isomorphism that allows to identify
\[
\pi _1 (E_2) = \pi _1 ( E_2 \times \check{o}_{E_2}) = \pi _1 (E_1(a,b)) = a^{-1} \pi _1 (E_1) \cap b^{-1} \pi _1 (E_1).
\]
For an arbitrary $r_1 \in R_1 = End(E_1)$, one has
\[
r_1 \pi _1 (E_2) = r_1 ( a^{-1} \pi _1 (E_1) \cap b^{-1} \pi _1 (E_1)) = a^{-1} ( r_1 \pi _1 (E_1)) \cap  b^{-1} ( r_1 \pi _1 (E_1)) =
\]
\[
= a^{-1} \pi _1 (E_1) \cap b^{-1} \pi _1 (E_1) = \pi _1 (E_2),
\]
so that $r_1 \in R_2 = End (E_2)$ and $R_1 \subseteq R_2$.
Similar considerations for the isomorphism $f^{-1} : (A_1, +) \rightarrow (A_2, +)$ of abelian surfaces yields $R_2 \subseteq R_1$, whereas $R_1 = R_2$.
The contradiction justifies that $A_1 = E_1 \times E_1$ and $A_2 = E_2 \times E_2$ are not birational for $End(E_1) = R_1 \neq R_2 = End(E_2)$.

Let ${\mathbb Q}( \sqrt{-d}$ be an imaginary quadratic number field and  $E_{-d,k} = {\mathbb C} / \mathcal{O}_{-d,k}$ be  the elliptic curve, whose fundamental group is the order $\mathcal{O}_{-d,k}$ of ${\mathbb Q}( \sqrt{-d})$ with conductor $k \in {\mathbb N}$.
 Lemma \ref{EndEqualPi1} has established the coincidence $R_{-d,k} = End (E_{-d,k}) = \pi _1 (E_{-d,k}) = \mathcal{O}_{-d,k}$.
 of the endomorphism ring and the fundamental group of $E_{-d,k}$.
 Therefore $R_{-d,m} = \mathcal{O}_{-d,m} \neq \mathcal{O}_{-d,n}  = R_{-d,n}$ for  different $m,n \in {\mathbb N}$ and the abelian surface
  $E_{-d,m} \times E_{-d,m}$ is not birational to the abelian surface $E_{-d,n} \times E_{-d,n}$.

\end{proof}

Let $\mathcal{O}_{-3,m}$ be the  order of the  imaginary quadratic number field ${\mathbb Q}( \sqrt{-3})$ with conductor $m \in {\mathbb N}$ and
$E_{-3,m} := {\mathbb C} / \mathcal{O}_{-3,m}$ be the elliptic curve with fundamental group $\mathcal{O}_{-3,m}$.
An arbitrary sequence $\{ k_n \} _{n=1} ^{\infty} \subset {\mathbb N} \setminus \{ 1 \}$ gives rise to a strictly increasing sequence
 $\{ m_n := \prod\limits _{j=1} ^n k_j \} \subset {\mathbb N}$ and a sequence
 \[
 I_2 : E_{-d, m_n} \times E_{-d, m_n} \longrightarrow E_{-d, m_{n-1}} \times E_{-d, m_{n-1}}
 \]
 of isogenies of non-birational abelian surfaces, as far as
 \[
\mathcal{O}_{-3, m_n}  = {\mathbb Z} + m_n \mathcal{O}_{-3} = {\mathbb Z} + m_{n-1} k_n \mathcal{O}_{-3} \varsubsetneq
  {\mathbb Z} + m_{n-1} \mathcal{O}_{-3} = \mathcal{O}_{-3, m_{n-1}}.
 \]
Applying Corollaries \ref{HirzebruchsPullBack}  and \ref{ModifiedHirzebruchsPullBack} with $\alpha =1$ and $m = m_n$, one obtains the following

\begin{corollary}   \label{NonBirationalEisensteinIsogenySeries}
For an arbitrary sequence $\{ k_n \} _{n=1} ^{\infty} \subset {\mathbb N} \setminus \{ 1 \}$ consider the strictly increasing sequence
 $\{ m_n := \prod\limits _{j=1} ^n k_j \} _{n=1} ^{\infty} \subset {\mathbb N}$ and the isogenies
\[
\lambda _n := I_2 : E_{-d, m_n} \times E_{-d, m_n} \longrightarrow E_{-d, m_{n-1}} \times E_{-d, m_{n-1}}.
\]

(i) Then
\[
\begin{diagram}
\node{( {\mathbb B} / \Gamma _{-3} ^{\rm Hir})'}  \node{( {\mathbb B} / \Gamma _1)'}  \arrow{w,t}{\lambda '_1}
\node{( {\mathbb B} / \Gamma _2)'}  \arrow{w,t}{\lambda '_2}   \node{\ldots}  \arrow{w,t}{\lambda '_3}
\end{diagram}
\]
is an infinite isogeny sequence of co-abelian,  torsion free, Picard modular toroidal compactifications over ${\mathbb Q}( \sqrt{-3})$, with mutually non-birational terms, infinitely increasing volume
\[
{\rm vol} ( {\mathbb B} / \Gamma _n) = \frac{8 \pi ^2}{3} m_n ^2 = \frac{8 \pi^2}{3} \prod\limits _{j=1} ^n k_j^2
\]
and infinitely increasing number of cusps
\[
h( {\mathbb B} / \Gamma _n) = 3 m_n + 1 = \prod\limits _{j=1} ^n k_j + 1.
\]

(ii) The infinite isogeny series
\[
\begin{diagram}
\node{( {\mathbb B} / \Gamma _{-3} ^{(1,4)})'}  \node{( {\mathbb B} / \Gamma _1^{(1,4)})'}  \arrow{w,t}{\lambda '_1}
\node{( {\mathbb B} / \Gamma _2^{(1,4)})'}  \arrow{w,t}{\lambda '_2}   \node{\ldots}  \arrow{w,t}{\lambda '_3}
\end{diagram}
\]
consists of  mutually non-birational, co-abelian, torsion free,  Picard modular  terms over ${\mathbb Q}( \sqrt{-3})$, with infinitely increasing volume
\[
{\rm vol} ( {\mathbb B} / \Gamma _n ^{(1,4)} ) = \frac{8 \pi ^2}{3} m_n^2 = \frac{8 \pi^2}{3} \prod\limits _{j=1} ^n k_j ^2
\]
and infinitely increasing number of cusps
\[
h ( {\mathbb B} / \Gamma _n ^{(1,4)} ) = 2 m_n +2 = \prod\limits _{j=1} ^n k_j +2.
\]
\end{corollary}

The constructions from Corollaries \ref{EisensteinIsogenySeriesWithBirationalTerms} and \ref{NonBirationalEisensteinIsogenySeries}  can be carried over with an arbitrary  proportional elliptic configuration $D \subset {\mathbb C}^2 / \mathcal{O}_{-d}^2$, defined over an imaginary quadratic number field
 ${\mathbb Q}( \sqrt{-d})$ of class number $CL( {\mathbb Q}( \sqrt{-d}) ) =1$.
 For instance, they are applicable to Holzapfel's example
\begin{equation}   \label{HolzapfelsPEC}
\begin{split}
D_{-1} ^{\rm Holz} = E_{-1} (1,0) + E_{-1}(0,1) + [ E_{-1} (-1,1) + Q_{03}] +  \\
 E_{-1} (-1-i,1) + E_{-1} (-1, 1-i) + [ E_{-1} (-i,1) + Q_{03}] \subset A_{-1}
 \end{split}
\end{equation}
from \cite{Ho01} with $A_{-1} = E_{-1} \times E_{-1}$, $E_{-1} = {\mathbb C} / ( {\mathbb Z} + {\mathbb Z}i)$, $Q_{03} = (Q_0,Q_3)$, $Q_0 = \check{o}_{E_{-1}}$,
$Q_3 = \frac{1+i}{2} + ( {\mathbb Z} + {\mathbb Z}i)$ and $( D_{-1} ^{\rm Holz}) ^{\rm sing} = \{ Q_{00}, Q_{03}, Q_{30} \}$.


\end{document}